\documentclass{amsart}
\newif\ifcd
\cdfalse

\usepackage{inputenc}
\usepackage[T1]{fontenc}
\usepackage{amsmath}
\usepackage{amsthm}
\usepackage{amsfonts}
\usepackage{amssymb}
\usepackage{enumerate}
\usepackage{mathrsfs}
\usepackage{graphicx}
\usepackage{hyperref}
\usepackage{subfigure}
\usepackage{tikz}
\usepackage{color}
\ifcd
\usetikzlibrary{matrix,calc,arrows,knots,intersections,cd}
\else
\usetikzlibrary{matrix,calc,arrows,intersections}
\fi
\usepackage{braids}
\usepackage{enumitem}
\usepackage{listings}
\usepackage{defs}

\title[Line arrangements with non-isomorphic group]{An arithmetic Zariski pair of line arrangements
with non-isomorphic fundamental group}
\author[E. Artal]{Enrique Artal Bartolo}
\address{Departamento de Matem\'aticas, IUMA\\ 
Universidad de Zaragoza\\ 
C.~Pedro Cerbuna 12\\ 
50009 Zaragoza, Spain} 
\email{artal@unizar.es} 

\author[J.I. Cogolludo]{Jos{\'e} Ignacio Cogolludo-Agust{\'i}n}
\address{Departamento de Matem\'aticas, IUMA\\ 
Universidad de Zaragoza\\ 
C.~Pedro Cerbuna 12\\ 
50009 Zaragoza, Spain} 
\email{jicogo@unizar.es} 

\author[B. Guerville]{Beno{\^i}t Guerville-Ball{\'e}}
\address{Insitut Joseph Fourier\\ 
UMR 5582 CNRS-UJF\\ 
100 rue des Math{\'e}matiques\\
BP 74 38 402 Saint-Martin-d'H{\`e}res Cedex, France} 
\email{benoit.guerville-balle@math.cnrs.fr}

\author[M. Marco]{Miguel Marco-Buzun{\'a}riz}
\address{Departamento de Matem\'aticas\\ 
Universidad de Zaragoza\\ 
C.~Pedro Cerbuna 12\\ 
50009 Zaragoza, Spain} 
\email{mmarco@unizar.es}

\thanks{First, second and fourth author are partially supported by
MTM2013-45710-C2-1-P. Third author is partially supported
by JSPS-MAE Sakura program}  



\begin{document}

\begin{abstract}
In a previous work, the third named author found a combinatorics of line arrangements whose realizations
live in the cyclotomic group of the fifth roots of unity and such that their non-complex-conjugate embedding
are not topologically equivalent in the sense that they are not embedded
in the same way in the complex projective plane. That work does not imply that the complements
of the arrangements are not homeomorphic. In this work we prove that the fundamental groups of the complements
are not isomorphic. It  provides the first example of a pair of Galois-conjugate plane curves such that
the fundamental groups of their complements are not isomorphic (despite the fact that they have 
isomorphic profinite completions).
\end{abstract}

\maketitle

\section*{Introduction}

The relationship between
topology and combinatorics is an important aspect in the study of hyperplane arrangements. 
As the main positive result, we have the fact that for a hyperplane arrangement
(say, in a complex projective space), the cohomology ring of the
complement is determined by the combinatorial data~\cite{orso:80}.
In~1994 G.~Rybnikov~\cite{ry:11-98} found a pair of line arrangements in 
$\mathbb{P}^2\equiv\mathbb{P}^2(\mathbb{C})$ such that their fundamental
groups are not isomorphic. In the long period between the
announcement and the publication of the result, Rybnikov's  statement
was reproved by three of the authors of this paper (with J.~Carmona) in~\cite{accm:03a},
using group extensions involving the Alexander invariant of the group and
a combinatorial generalization of some Rybnikov's ideas by the fourth named author~\cite{mmarco-resonance}.
Note that Rybnikov's arrangements cannot be the complexification of a real arrangement.

Some years later, a new example of arrangements sharing the combinatorics but not all of the 
topological properties was found~\cite{accm}. This example was weaker than Rybnikov's one in
some aspects and somewhat stronger in other ones. The author distinguished the two arrangements
using a braid monodromy invariant developed in \cite{acc:01a}; this invariant is able to distinguish
two pairs formed by the complex projective plane and an algebraic curve but it does not give 
further information on the topology of the complement (in particular, on the fundamental group).
These arrangements are defined over $\mathbb{Q}(\sqrt{5})$ (in particular they are the complexification
of a real arrangement) and, moreover, they have Galois-conjugate equations in this field.
As a consequence the fundamental groups of their complements have the same profinite completion
and they share all the topological properties of algebraic nature (quite a lot!).

Recently, the third named author~\cite{bgb:14a} found another combinatorics with distinct topological realizations.
He used a new invariant developed in~\cite{ea:tls,afg:14} computing the image of some special homological
cycle in the complement by some associated character; this computation is doable using the injection
of the boundary of a regular neighborhood of the arrangement in the complement, done in~\cite{fgm}.
In this case the realizations are defined over the cyclotomic group of the fifth roots of unity
and only the topology of the pairs is distinguished by this invariant as in the latter example.
Up to our knowledge, these are (essentially) the only known examples of line arrangements
with the same combinatorics and distinct topology.

The main result of this paper is that in fact, the fundamental group of the realizations of
the above combinatorics are not isomorphic (of course, for the realizations which are not complex-conjugate).
We have used the same method which was already successful in~\cite{accm:03a}.
The notion of combinatorics can be generalized for arbitrary plane algebraic curves, see~\cite{ea:jag}:
a pair of curves with the same combinatorics and distinct topology is called a \emph{Zariski pair}.
A Zariski pair is said to be \emph{arithmetic} if its members have Galois-conjugate equations over a number
field. As we stated before, it is hard to check if a Zariski pair candidate is an actual Zariski pair;
as an example there are arithmetic Zariski pairs candidate communicated to the authors by Fan for which 
the known invariants failed to check if they are actual arithmetic Zariski pairs.

The arithmetic Zariski pairs are related to conjugate varieties. The fact that conjugate varieties
may have distinct topology was proved by Serre~\cite{se:64}; other examples were found by
Abelson~\cite{ab:74} and more recently by many authors~\cite{charles:09,rajan,gd-tg}. 
For the case of curves, the first example of an arithmetic Zariski pair is in~\cite{acc:01b} (a curve
of degree~12). The complete list of such Zariski pairs for sextics was done by Shimada~\cite{shi:08,shi:09};
a longer list (replacing homeomorphism by a special type of diffeomorphism) was given by
Degtyarev~\cite{degt:08}. In all these examples, nothing is said about the topology of the complement;
in fact, in~\cite{act:08} a Zariski pair with homeomorphic complements is given. Hence,
our example is the first arithmetic Zariski pair with non-isomorphic fundamental group.

The paper is organized as follows. In~\S\ref{sec-combinatorics}, we introduce the combinatorics 
$\AG$ studied
in the paper. In~\S\ref{sec-rigid}, we study the homological rigidity of this combinatorics. The
fact that $\AG$ is homologically rigid and has no non-trivial automorphism implies that given
two realizations $\mathcal{A}_i$, $i=1,2$, any isomorphism 
$\Phi:\pi_1(\mathbb{P}^2\setminus\mathcal{A}_1)\to\pi_1(\mathbb{P}^2\setminus\mathcal{A}_2)$ 
is \emph{$\pm$ the identity in homology}. In \S\ref{sec-alex}, we introduce the truncated Alexander
invariants. In \S\ref{sec-fg}, we compute the fundamental groups using the wiring diagrams
computed in~\cite{bgb:14a} and we explain how to compute the truncated Alexander invariants.
We show that no isomorphism
$\Phi:\pi_1(\mathbb{P}^2\setminus\mathcal{A}_1)\to\pi_1(\mathbb{P}^2\setminus\mathcal{A}_2)$ 
can induced the identity on homology and using a conjugate arrangement, we do the same for its opposite,
proving the result. In the Appendix~\ref{sec-code}, we explain the \texttt{Sagemath} code that makes 
the heavy computations.

\section{Combinatorics and realizations}
\label{sec-combinatorics}
In this section a line combinatorics will be described together with several possible realizations 
on the complex projective plane.

\subsection{Line combinatorics}
\mbox{}

For the sake of completeness we will recall the definitions of line combinatorial types.
\begin{dfn}\label{def-comb}
A \emph{line combinatorial type} (or simply a \emph{(line) combinatorics}) is a pair
$\scc:=(\cL,\cP)$, where $\cL$ is a finite set and $\cP\subset \cP(\cL)$, satisfying:
\begin{enumerate}
\item 
For all $P\in\cP$, $\#\cP\geq 2$;
\item
For any $L_1,L_2\in \cL$, $L_1\neq L_2$, $\exists!P\in \cP$ such that $L_1,L_2\in P$. 
\end{enumerate}
An {\it ordered combinatorial type} $\scc^{\text{ord}}$ 
is a combinatorial type where $\cL$ is an ordered set.
\end{dfn}

\begin{ntc}
\label{not-comb-type}
Given a combinatorial type $\scc$, the multiplicity $m_P$ of $P\in\cP$ is the number 
of elements $L\in\cL$ such that $P\in L$; note that $m_P\geq 2$. In order to mimic the
situation in the realizations of a combinatorics, we will write $P\in L$
when the \emph{line}~$P$ is in the \emph{point}~$P$; this relation will also be indicated
as $L<P$.
\end{ntc}

\subsection{The combinatorics \texorpdfstring{$\AG$}{G91}}
\mbox{}

The combinatorics (and realizations) described here were originally presented in~\cite{bgb:14a}.
In what follows, a more geometrically insightful approach will be presented. This will be helpful
in order to describe the required invariants to compare their fundamental groups.

Let us begin with an ordered set of four points $P_1,\dots,P_4\in\mathbb{P}^2$ in general position
(red points in \autoref{fig:12lines}). 
Note that such a set is unique up to projective automorphism:
$$
P_1,\dots,P_4=[1:1:1],[1:-1:1],[1:-1:-1],[1:1:-1].
$$

The first four lines of our arrangement are the lines $L_1,\dots,L_4$ of the quadrangle defined by 
$P_1,\dots,P_4$:
$$
L_1:x-z=0,\quad L_2:x+y=0,\quad L_3:x+z=0,\quad L_4:x-y=0.
$$
The two diagonals of the quadrangle should also be considered: $L_5$ is the line $\overline{P_1 P_3}$
and $L_6$ is the line $\overline{P_2 P_4}$. Their equations are:
$$
L_5:y-z=0,\quad L_6:y+z=0.
$$
After choosing two points $P_5$ (resp. $P_6$) on $L_5$ (resp. $L_6$) one can consider a set of four lines 
$L_9:=\overline{P_1P_6},L_{10}:=\overline{P_2P_5},L_7:=\overline{P_3P_6},L_8:=\overline{P_4P_5}$ 
as shown in Figure~\ref{fig:12lines}. The intersection of these lines with the original lines $L_1,...,L_4$
define a set of points as follows: $Q_i=L_i\cap L_{i+6}$, $i=1,...,4$.

In principle, the newly defined points $Q_1,...,Q_4$ need not be aligned. The combinatorics $\AGt$ 
is determined by a particular choice of $P_5$ and $P_6$ for which $Q_1,...,Q_4$ belong in a line, say~$L_{11}$.

% \begin{figure}[ht]
% \begin{tikzpicture}[scale=1,vertice/.style={draw,circle,fill,minimum size=0.25cm,inner sep=0}]
% \tenlines
% \eleventhline
% \nodos
% \end{tikzpicture}
% \caption{The combinatorics~$\AGt$.}
% \label{fig:11lines}
% \end{figure}

The existence of such points $P_5$ and $P_6$ in $\bp^2$ satisfying the condition given in $\AGt$ 
is not unique (as it was proved in~\cite{bgb:14a}). 
In fact, there are four choices leading to the following realizations:
\begin{gather*}
L_9: x + \left(\xi^{2} + \xi\right) y + \overline{\left(\xi^{2} + \xi\right)} z=0,\quad
L_{10}: x - \left(\overline{\xi}^{2} +\xi\right) y + \left(\overline{\xi}+\xi^{2}\right) z=0,\\
L_7: x - \overline{\left(\xi^{2} + \xi\right)} y - \left(\xi^{2} +\xi\right) z=0,\quad
L_8: x + \left(\overline{\xi}+\xi^{2}\right) y - \left(\overline{\xi}^{2} +\xi\right) z=0,
\end{gather*}
where $\xi$ is a primitive fifth-root of unity. 

Finally, one can also obtain equations of the line $L_{11}$ joining $Q_1,\dots,Q_4$:
$$
L_{11}:5 x + \left(1+2\xi+3\xi^2-\overline{\xi}^2\right) y - 
\left(2+4\xi+\xi^2+3\overline{\xi}^2\right) z=0.
$$
For further use, the four realizations parametrized by the primitive fifth-roots of unity
will be denoted by~$(\cA')^\xi$.

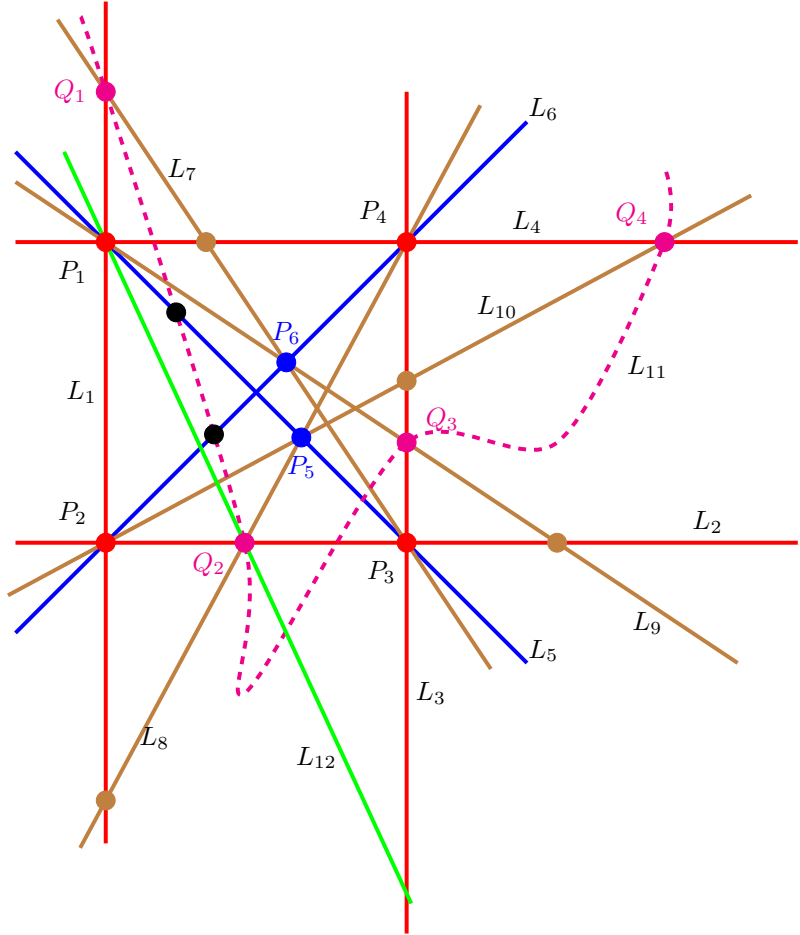
\begin{figure}[ht]
\begin{tikzpicture}[scale=2,vertice/.style={draw,circle,fill,minimum size=0.25cm,inner sep=0}]
\def\tenlines{%vertices del cuadrilatero
%%%%%%%%%%%%%%%
\node[label={225:$P_1$}] (P1) at (-1,1) {};
\node[label={135:$P_2$}] (P2) at (-1,-1) {};
\node[label={[xshift=3pt]225:$P_3$}] (P3) at (1,-1) {};
\node[label={135:$P_4$}] (P4) at (1,1) {};
%%%%%%%%%%%%%%%

%Puntos triples de las diagonales
%%%%%%%%%%%%%%%
\node[label={272:\color{blue}{$P_5$}}] (P5) at ($(P1.center)!.65!(P3.center)$) {};
\node[label={90:\color{blue}{$P_6$}}]  (P6) at ($(P2.center)!.6!(P4.center)$) {};
%%%%%%%%%%%%%%%

%aristas del cuadrilatero
%%%%%%%%%%%%%%%
\draw[name path=L1,line width=1.5pt,color=red] ($(P1.center)!-.8!(P2.center)$)--($(P2.center)!-1!(P1.center)$);
\draw[name path=L2,line width=1.5pt,color=red] ($(P2.center)!-.3!(P3.center)$)--($(P3.center)!-1.3!(P2.center)$);
\draw[name path=L3,line width=1.5pt,color=red] ($(P3.center)!-1.3!(P4.center)$)--($(P4.center)!-.5!(P3.center)$);
\draw[name path=L4,line width=1.5pt,color=red] ($(P4.center)!-1.3!(P1.center)$)--($(P1.center)!-.3!(P4.center)$);
%%%%%%%%%%%%%%%

%Diagonales
%%%%%%%%%%%%%%%
\draw[name path=L5,line width=1.5pt,color=blue] ($(P1.center)!-.3!(P3.center)$)--($(P3.center)!-.4!(P1.center)$);
\draw[name path=L6,line width=1.5pt,color=blue] ($(P2.center)!-.3!(P4.center)$)--($(P4.center)!-.4!(P2.center)$);
%%%%%%%%%%%%%%%

%Rectas de los vertices a los puntos triples de las diagonales
%%%%%%%%%%%%%%%
\draw[name path=M1,line width=1.5pt,color=brown] ($(P1.center)!-.5!(P6.center)$)--($(P6.center)!-2.5!(P1.center)$);
\draw[name path=M2,line width=1.5pt,color=brown] ($(P2.center)!-.5!(P5.center)$)--($(P5.center)!-2.3!(P2.center)$);
\draw[name path=M3,line width=1.5pt,color=brown] ($(P3.center)!-.7!(P6.center)$)--($(P6.center)!-1.9!(P3.center)$);
\draw[name path=M4,line width=1.5pt,color=brown] ($(P4.center)!-.7!(P5.center)$)--($(P5.center)!-2.1!(P4.center)$);
%%%%%%%%%%%%%%%

%Puntos triples en el cuadrilatero
%%%%%%%%%%%%%%%
\path [name intersections={of=L1 and M3,by=Q1}];
\path [name intersections={of=L2 and M4,by=Q2}];
\path [name intersections={of=L3 and M1,by=Q3}];
\path [name intersections={of=L4 and M2,by=Q4}];
%%%%%%%%%%%%%%%

%Puntos dobles en el cuadrilatero
%%%%%%%%%%%%%%%
\path [name intersections={of=L1 and M4,by=R1}];
\path [name intersections={of=L2 and M1,by=R2}];
\path [name intersections={of=L3 and M2,by=R3}];
\path [name intersections={of=L4 and M3,by=R4}];
%%%%%%%%%%%%%%%
}

\def\eleventhline{%"Recta" que une los puntos triples
%%%%%%%%%%%%%%%
\draw[name path= M11,dashed,line width=1.5pt,color=magenta] plot [smooth] coordinates {($(Q1)+(-1/6,3/6)$) (Q1) (Q2) ($(Q2)+(0,-1)$) (Q3) ($(Q3)+(1,0)$)  (Q4) ($(Q4)+(0,.5)$)};
\path [name intersections={of=L5 and M11,by=Q511}];
\path [name intersections={of=L6 and M11,by=Q611}];
%%%%%%%%%%%%%%%
}

\def\twelvethline{
%Recta que rompe la simetria
%%%%%%%%%%%%%%%
\draw[name path=N,line width=1.5pt,color=green] ($(P1.center)!-.3!(Q2.center)$)--($(Q2.center)!-1.2!(P1.center)$);
%%%%%%%%%%%%%%%
}

\def\nodos{
\node[vertice,color=red] at (P1) {};
\node[vertice,color=red] at (P2) {};
\node[vertice,color=red] at (P3) {};
\node[vertice,color=red] at (P4) {};
\node[vertice,color=blue] at (P5) {};
\node[vertice,color=blue] at (P6) {};
\node[vertice,color=magenta,label={180:\color{magenta}{$Q_1$}}] at (Q1) {};
\node[vertice,color=magenta,label={183:\color{magenta}{$Q_2$}}] at (Q2) {};
\node[vertice,color=magenta,label={10:\color{magenta}{$Q_3$}}] at (Q3) {};
\node[vertice,color=magenta,label={135:\color{magenta}{$Q_4$}}] at (Q4) {};
\node[vertice,color=brown] at (R1) {};
\node[vertice,color=brown] at (R2) {};
\node[vertice,color=brown] at (R3) {};
\node[vertice,color=brown] at (R4) {};
\node[vertice,color=black] at (Q511) {};
\node[vertice,color=black] at (Q611) {};
}
\def\nombres{
\node[left] at ($.5*(P1.center)+.5*(P2.center)$){$L_1$};
\node[above] at ($2*(P3.center)-1*(P2.center)$){$L_2$};
\node[right] at ($1.5*(P3.center)-.5*(P4.center)$){$L_3$};
\node[above] at ($-.4*(P1.center)+1.4*(P4.center)$){$L_4$};
\node[above right=-3pt] at ($-.4*(P1.center)+1.4*(P3.center)$){$L_5$};
\node[above right=-3pt] at ($-.4*(P2.center)+1.4*(P4.center)$){$L_6$};
\node[above right=-3pt] at ($-1*(P3.center)+2*(P6.center)$){$L_7$};
\node[below right=-5pt] at ($-1.5*(P4.center)+2.5*(P5.center)$){$L_8$};
\node[below] at ($-2*(P1.center)+3*(P6.center)$){$L_9$};
\node[above=2pt] at ($-1*(P2.center)+2*(P5.center)$){$L_{10}$};
\node[above=2pt] at ($1.3*(P3.center)+1.3*(P4.center)$){$L_{11}$};
\node[above=2pt] at ($1.5*(P3.center)+1.1*(P2.center)$){$L_{12}$};
}
\tenlines
\eleventhline
\twelvethline
\nodos
\nombres
\end{tikzpicture}
\caption{The combinatorics~$\AG$.}
\label{fig:12lines}
\end{figure}

The ordered combinatorics~$\AGt$ has a non-trivial automorphism group,
generated by the product of the cyclic permutations of $(L_1,\dots,L_4)$ and $(L_7,\dots,L_{10})$,
and the transposition of $(L_5,L_6)$, while $L_{11}$ is invariant.
Note that the action of this group of automorphisms can be realized by projective transformations
sending cyclically $(\cA')^{\xi^i}$ to~$(\cA')^{\xi^{i+1}}$.

The final combinatorics $\AG$ with trivial automorphism group can be obtained from $\AGt$ 
by adding a new line $L_{12}$ joining $P_1$ and~$Q_2$.
The four realizations of $\AG$ will be denoted by $\cA^\xi$.
The line $L_{12}$ has equations:
\begin{gather*}
L_{12}: x - \left(1+\overline{\xi}\right) y + \overline{\xi} z=0.
\end{gather*}

In terms of \autoref{def-comb}, the combinatorics is defined 
$\mathcal{L}=\{1,\dots,12\}$ and $\mathcal{P}$ below:
\begin{gather*}
\left\{1, 4, 5, 9,12\right\},
\left\{1, 2, 6,10\right\},
\left\{2, 3, 5, 7\right\}, 
\left\{3, 4, 6,8\right\},
\left\{2,8,11, 12\right\},
\left\{1,7,11\right\},
\left\{3, 9, 11\right\},
\left\{4,10, 11\right\},
\\
\left\{5,8,10\right\},  \left\{6, 7,9\right\},  \left\{5, 6\right\}, \left\{5,11\right\}, \left\{6,11\right\},
\left\{1, 3\right\},\left\{2,4\right\}, 
\left\{1,8\right\},  \left\{2,9\right\}, \left\{3,10\right\},\left\{4,7\right\}, \\
\left\{7,8\right\},\left\{7,10\right\}, \left\{8, 9\right\},   \left\{9,10\right\}, 
\left\{3, 12\right\}, \left\{6,12\right\},   \left\{7,12\right\},\left\{10, 12\right\}
\end{gather*}

\section{Homological rigidity}\label{sec-rigid}

Let $\mathscr C=(\mathcal L,\mathcal P)$ be a line combinatorics as presented in 
\S\ref{sec-combinatorics}. Let $\mathbb{Z}^{\mathcal{L}} $ be a free abelian group
with basis $\{x_L\mid L\in\mathcal{L}\}$; the dual basis in $\left(\mathbb{Z}^{\mathcal{L}}\right)^*$
is denoted by $\{y_L\mid L\in\mathcal{L}\}$. We define $H_1=H_1^{\mathscr{C}}$ as the
quotient of $\mathbb{Z}^{\mathcal{L}} $ by the principal submodule generated by
$\sum_{L\in\mathcal{L}} x_L$; in the same way we denote its dual
by $H^1=H^1_{\mathscr{C}}$. Note that $H^1$ is the kernel of the augmentation
morphism $\varepsilon: \left(\mathbb{Z}^{\mathcal{L}}\right)^*\to\mathbb{Z}$
associated to the above basis.
These lattices are naturally isomorphic to the first homology and cohomology groups of the complement 
in $\PP^2$ of any realization of $\mathscr C$. 

To avoid notation overload the class of $x_L$ in $H_1$ is still denoted by $x_L$.
For each $P\in\mathcal{P}$, let us denote $x_P:=\sum_{L<P} x_L$.
Let $ H_2:=H_2^{\mathscr{C}}$
the subgroup of $H_1\wedge H_1$ generated by 
\begin{equation}
\label{eq-H2}
 \left\{x_{L,P}:=x_{L}\wedge x_P\mid L< P\in \mathcal{P}\right\}.
\end{equation}
Its dual $ H^2:= H^2_{\mathscr C}$ is the quotient of $H^1\wedge H^1$ 
by the subgroup generated by
\begin{equation*}
 \left\{y_{L}\wedge y_{L'}+ y_{L'}\wedge y_{L''}+ y_{L''}\wedge y_{L} \mid L,L',L''< P\in \mathcal{P}\right\}
\end{equation*}
as part of the Orlik-Solomon algebra. As above, these groups
are naturally isomorphic 
to the second homology and cohomology groups of the complement in $\PP^2$ of any realization of $\mathscr C$. 
Besides the automorphisms of the combinatorics, there are some geometrical automorphisms 
that will be considered here and play an essential role in this theory.

\begin{dfn}
With the previous notation, any automorphism of $H^{\mathscr{C}}_1$ inducing a morphism of 
$H_2^{\mathscr{C}}$ will be referred to as an \emph{admissible} automorphism of the combinatorics. 
The group of admissible automorphisms will be denoted by $\Adm_{\mathscr C}$.
\end{dfn}

\begin{obs}\label{obs-admis}
Admissible automorphisms are also closely related to isomorphisms of fundamental groups
of realizations of combinatorics. It is a consequence of \autoref{obs-combi} that
if $\mathcal{A}_1,\mathcal{A}_2$ are two realizations of~$\mathscr{C}$ and
$\varphi:\pi_1(\mathbb{P}^2\setminus\mathcal{A}_1)\to\pi_1(\mathbb{P}^2\setminus\mathcal{A}_2)$
is an isomorphism, under the above identifications, then its induced automorphism 
$\varphi_*:H^{\mathscr{C}}_1\to H^{\mathscr{C}}_1$ is admissible.
\end{obs}

% Again, note that any homeomorphism between two realizations of the same combinatorics 
% induces an admissible automorphism in homology.

\begin{dfn}[\cite{mmarco-resonance}]
A combinatorics $\mathscr C$ is called \emph{homologically rigid} if 
$\Adm_{\mathscr C}=\pm 1_{H^{\mathscr C}_1}\times\Aut_{\mathscr C}$.
\end{dfn}

Homological triviality is fundamental for the study of certain fine invariants of the topology 
of different realizations of~$\mathscr C$. Checking the homological triviality of a combinatorics
is a subtle combinatorial property.

The main tools for the study of admissible automorphisms and homological rigidity
are the strata of a certain 
stratification of $H_{\mathscr C}^1$ called resonance varieties, which will be described next.

\subsection{The resonance varieties of a combinatorics}
\mbox{}

For any $\omega\in H^1=H^1_{\mathscr C}$, one can define a complex $(H^\bullet,\wedge \omega)$ given by 
the wedge product as follows:
$$
\array{ccccccc}
H^0=\ZZ & \longrightmap{\wedge \omega} & H^1 & \longrightmap{\wedge \omega} & H^2 & \to & 0\\
&& \sigma & \mapsto & \sigma \wedge \omega. &&
\endarray
$$
We are interested in the first cohomology of this complex $H^1(H^\bullet,\wedge \omega)$.

\begin{dfn}
The \emph{$k$-th resonance variety} of $\mathscr C$ is defined as
$$\mathcal R_{k,\mathscr C}:=\{\omega\in H^1 \mid 
\rank H^1(H_{\mathscr C}^\bullet,\wedge \omega)\geq k\}\subset H^1.$$
Analogously, one can define $\mathcal R_{k,\mathscr C,\KK}$ over a field $\KK$ when considering the 
complex over~$\KK$ and stratifying the vector space $H^1\otimes \KK$ by the dimension of the 
associated cohomology.
\end{dfn}

As expected, the resonance varieties $\mathcal R_{k,\mathscr C,\KK}$ of a combinatorics 
coincide with the resonance varieties over $\KK$ of any realization of~$\mathscr C$.

Each $k$-th resonance component over $\ZZ$ is determined by a subcombinatorics with a certain structure 
called \emph{combinatorial pencils} (cf.~\cite{mmarco-resonance}) or multinets 
(cf.~\cite{Falk-Yuzvinsky-multinets}) with $k+2$ fibers. The structure of such combinatorial pencils with
$k+2$ fibers mimics the combinatorial properties of a line arrangement which is the union of $k+2$ fibers
of a pencils of curves. Such subarrangements will be referred to here using the typical pencil they describe. 
For instance \emph{multiple point}, \emph{Ceva-type} (the combinatorics of the six lines in a generic pencil 
of smooth conics) or \emph{Hesse type} (for the combinatorics of the 12 lines in an arrangement of smooth 
cubics based at their 9 inflection points). 

Let us denote by $S$ a certain 
subcombinatorics of $\mathscr{C}$ (obtained as a subset of lines and their incidence relations)
forming a combinatorial pencil. There is exactly one irreducible component in
the resonance varieties of $S$ which is not contained in the resonance varieties of a subcombinatorics. 
Such component will be denoted by 
$H_S$ and we refer to as the \emph{resonance component associated to} $S$. Note that its dimension is $k+1$, if $k+2$ is the number of fibers of the 
combinatorial pencil.

\begin{dfn}
Three subcombinatorics $S_1, S_2, S_3$ of combinatorial pencils in a line combinatorics are said to form a 
\emph{triangle} if 
$$\codim \bigcap_i H_{S_i}=\sum_i\codim H_{S_i}-1.$$
\end{dfn}

The following result is obtained in~\cite[Lemma 6]{mmarco-resonance}.

\begin{lema}
\label{lem-triangles}
Let $\rho\in \Adm_{\mathscr C}$ be an admissible automorphism of $H_1$. Then its dual $\rho^* $ preserves triangles.
\end{lema}

As depicted in Lemma~\ref{lem-triangles}, the resonance components (their dimension and the number of 
triangles they belong to) can impose conditions on admissible automorphisms.
The following is an immediate consequence of~\cite{Yuzvinsky-new-bound}.

\begin{cor}
\label{cor-multiple-point}
Let $\mathscr C$ be a realizable combinatorics and $R$ a $k$-th resonance component with $k>2$. 
%If $\rho:H^1\to H^1$ is an admissible automorphism, 
Then $R$ is associated with the combinatorial pencil 
of a multiple point of multiplicity~$m=k+2>4$.
\end{cor}

This has the following consequence on an admissible automorphism $\rho$. Let $P_m$ be a point on $\mathcal{P}$ 
of multiplicity $m$, which is a subcombinatorics of $\mathscr C$ and denote by $H_{P_m}\subset H^1$ its 
associated resonance component of dimension $k$. Then any admissible automorphism $\rho:H_1\to H_1$ must be 
such that $\rho^* (H_{P_m})=H_{Q_m}$ for another $Q_m\in \mathcal{P}$ of multiplicity~$m$.

This results in certain combinatorial conditions on the images of the special basis in~$H_1$. 
If enough of these conditions concur, one can eventually be able to state homological rigidity.
The idea behind this concept comes from the work of Rybnikov (cf.~\cite{ry:11-98,accm:03a}).

\subsection{The homological rigidity of \texorpdfstring{$\AG$}{G91}}
\mbox{}

In our example, it is possible to compute that the combinatorial pencils contained in the combinatorics 
$\AG$ described in \S\ref{sec-combinatorics}: the 10 multiple points and 15 Ceva-type subarrangements. 
Only one of the 25 combinatorial pencils has 5 fibers (the quintuple point).

% Let us denote by $\Delta_S$ the number of triangles in the subcombinatorics $S$ and by 
% $\Delta_{P,S}$ the number of triangles containing the (combinatorial pencil associated with a certain) point~$P$.

The number of triangles $\Delta_S$ in a given subcombinatorics $S$ as well as the number of triangles 
$\Delta_{P_1,S}$ containing the quintuple point $P_1$ are recorded in Table~\ref{tab-combinatorics}.

%\vspace*{14pt}
\begin{table}
\centering
\begin{tabular}{|c|c|c|c|c|}
\hline 
$i$& subcombinatorics $S_i$ & $\dim H_S$ & $\Delta_S$ & $\Delta_{S,P_1}$ \\
\hline 
1 & 1, 7, 11 & 2 & 18 & 7 \\
2 & 3, 9, 11 & 2 & 22 & 8 \\
3 & 4, 10, 11 & 2 & 21 & 7 \\
4 & 5, 8, 10 & 2 & 24 & 7 \\
5 & 6, 9, 7 & 2 & 16 & 6 \\

6 & 1, 2, 6, 10 & 3 & 53 & 12 \\
7 & 2, 3, 5, 7 & 3 & 49 & 13 \\
8 & 2, 8, 11, 12 & 3 & 57 & 15 \\
9 & 4, 3, 6, 8 & 3 & 50 & 12 \\

10 & 1, 4, 5, 9, 12 & 4 & 91 & 91 \\

11 & 1, 2, 3, 4, 5, 6 & 2 & 24 & 8 \\
12 & 1, 2, 4, 6, 8, 12 & 2 & 24 & 8 \\
13 & 1, 2, 4, 10, 11, 12 & 2 &20 & 7 \\
14 & 1, 2, 5, 6, 7, 9 & 2 & 14 & 7 \\
15 & 1, 2, 5, 7, 11, 12 & 2 & 14 & 7 \\
16 & 1, 2, 5, 8, 10, 12 & 2 & 20 & 8 \\
17 & 1, 3, 5, 7, 9, 11 & 2 & 14 & 7 \\
18 & 1, 4, 5, 6, 8, 10 & 2 & 19 & 6 \\
19 & 2, 3, 4, 5, 8, 12 & 2 & 20 & 8 \\
20 & 2, 3, 5, 6, 8, 10 & 2 & 14 & 0 \\
21 & 2, 3, 5, 9, 11, 12 & 2 & 18 & 9 \\
22 & 2, 4, 6, 8, 10, 11 & 2 & 15 & 0 \\
23 & 3, 4, 5, 6, 7, 9 & 2 & 12 & 6 \\
24 & 3, 4, 8, 9, 11, 12 & 2 & 13 & 7 \\
25 & 4, 5, 8, 10, 11, 12 & 2 & 15 & 7 \\
\hline
\end{tabular} 
\vspace*{14pt}
\caption{Combinatorics of triangles in $\AG$}
\label{tab-combinatorics}
\end{table}

Note that the subcombinatorics $S_i$, $i=1,...,5$ correspond to triple points, 
whereas $S_i$, $i=6,...,9$ correspond to quadruple points, and $S_{10}$ is the combinatorics of the quintuple
point. Finally, the remaining $S_i$, $i=11,...,25$ correspond to the combinatorics of Ceva-type arrangements.

As a consequence of Table~\ref{tab-combinatorics} one has the following.

\begin{prop}\label{prop-hom-rig}
The realizable combinatorics $\AG$ from{\rm~\S\ref{sec-combinatorics}} is homologically rigid.
Moreover, the only admissible automorphisms of $H^1_{\AG}$ are $\pm 1_{H^1_{\AG}}$.
\end{prop}

\begin{proof}
By Lemma~\ref{lem-triangles} and Corollary~\ref{cor-multiple-point}, triangles containing the 
quintuple point are also preserved. Since the last two columns of subcombinatorics of dimension 3 (the 
quadruple points) are all different, such resonance components are invariant by $\rho$. Analogously, resonance
components of triple points are also $\rho$-invariant. We use this information to prove
the property.

Let $\rho:H_1\to H_1$ be an admissible automorphism. Let us choose any lift~$\tilde{\rho}$ fitting in the following
commutative diagram
\ifcd
\begin{equation*}
\begin{tikzcd}[column sep=1em]
\mathbb{Z}^{\mathcal{L}}\arrow[rr,"\tilde{\rho}"]\arrow[d]&&\mathbb{Z}^{\mathcal{L}}\arrow[d]\\
H_1^{\mathscr{C}}\arrow[rr,"\rho"]&&H_1^{\mathscr{C}}.
\end{tikzcd} 
\end{equation*}
\else
\begin{equation*}
\begin{tikzpicture}[description/.style={fill=white,inner sep=2pt},baseline=(current bounding box.center)]
\matrix (m) [matrix of math nodes, row sep=2.5em,
column sep=1em, text height=1.5ex, text depth=0.25ex]
{\mathbb{Z}^{\mathcal{L}}&&\mathbb{Z}^{\mathcal{L}}\\
H_1^{\mathscr{C}}&&H_1^{\mathscr{C}}\\};
\path[->,>=angle 90](m-1-1) edge node[above] {$\tilde{\rho}$} (m-1-3);
\path[->,>=angle 90](m-1-1) edge (m-2-1);
\path[->,>=angle 90](m-2-1) edge node[above] {$\rho$} (m-2-3);
\path[->,>=angle 90](m-1-3) edge (m-2-3);
\end{tikzpicture}
\end{equation*}
\fi
Let $A$ be a matrix of $\tilde{\rho}$ in the natural basis. Note that the columns of~$A$ are only well defined
up to addition of a multiple of~$v:=(1,\dots,1)$ and that ${}^t A$ is the matrix of a lift of $\rho^*$. 
Let $(i_1,\dots,i_r)$ be a multiple point. Since $\rho^*$ fixes the resonance component associated to this point, 
with basis $y_{i_2}-y_{i_1},\dots,y_{i_r}-y_{i_1}$, we deduce the following fact. Consider the submatrix of $A$ 
given by the rows $i_1,\dots,i_r$ and the columns distinct from $i_1,\dots,i_r$; then, all the rows of this matrix 
are equal. This comes from the fact that $\rho^*(y_{i_2}-y_{i_1})$ is in the subspace spanned by $y_{i_1},\dots,y_{i_r}$.

It is not hard to prove that after imposing the above condition, the matrix $A$ has the following property: each 
column has constant entries outside the diagonal. Since the matrix $A$ can be transformed by adding multiples of~$v$ 
to the columns, we can assume that $A$ is diagonal.

Since $\rho$ is a group automorphism, one has $\rho(x_L)=\varepsilon_L x_L$, where $\varepsilon_L\in \{-1,1\}$. 
Moreover, the condition $\rho(\sum_L x_L)=0$ implies that all $\varepsilon_L$ are equal and the result follows.
\end{proof}

\section{The truncated Alexander Invariants}
\label{sec-alex}

In this section the truncated Alexander invariants introduced in~\cite{accm:03a} will be recalled.
Let $G$ be a group, and let $H=H(G):=G/G'$ be its abelianization and consider $M=M(G):=G'/G''$ as an abelian group. 
The conjugation action of $G$ on $G'$, namely, $g\cdot h\mapsto ghg^{-1}$ induces an action of $H$ on $M$. 
This action extends by linearity to $\Lambda:=\mathbb{Z}[H]$, the group algebra of $H$.
The \emph{Alexander invariant} of $G$ is the abelian group $M$ together with the $\Lambda$-module structure.

Note that $\Lambda$ is, in general, not a PID and hence it is not easy to give complete invariants for the 
Alexander invariant. One standard way to approximate the structure of $M$ is by considering its truncation
with respect to a special ideal. A standard way to do this is by means of $\mathfrak{m}\subset\Lambda$ the 
augmentation ideal of $\Lambda$, i.e. the kernel of the map $\Lambda\to\mathbb{Z}$ defined by $h\mapsto 1$, 
$\forall h\in H$ (despite the notation, note that $\mathfrak{m}$ is not a maximal ideal). 

\begin{dfn}
\label{dfn-truncated}
The \emph{truncated Alexander invariant} of $G$ of order $k$ associated with $M=M(G)$ and the augmentation 
ideal $\mathfrak{m}\subset \Lambda$ is the quotient 
$M_k:=M/\mathfrak{m}^k M=M\otimes_{\Lambda}\Lambda/\mathfrak{m}^k$, $k\geq 1$, with its 
$\Lambda/\mathfrak{m}^k$-module structure.
\end{dfn}

These invariants are related to the Chen groups $\tilde{\gamma}_k(G)$ of $G$, that is, the lower central 
series of $G/G''$, the maximal metabelian quotient of $G$. They may alternatively be defined as 
$\tilde{\gamma}_{k+2}(G):=\ker(\varphi_k:G'\to M_k)$ where $\varphi_k$ is the quotient map.

\begin{ntc}
Given two elements $p,q\in G'$, the notation $p\overset{k}{\equiv} q$ will be used meaning equality 
as elements in~$M_k$, that is, $\varphi_k(p)=\varphi_k(q)\in M_k$.
\end{ntc}

\begin{ejm}
\label{ejem-truncated}
Let $\mathcal{A}$ be a line arrangement in $\mathbb{P}^2$, and let
$G_\mathcal{A}:=\pi_1(\mathbb{P}^2\setminus\mathcal{A})$.
We will refer to $M(\mathcal{A}):=M(G_\mathcal{A})$ (resp. $M_k(\mathcal{A}):=M_k(G_\mathcal{A})$) as
the Alexander invariant (resp. truncated Alexander invariants) of~$\mathcal{A}$.

If~$\mathcal{A}=\{L_0,L_1,\dots,L_\ell\}$, then $G_\mathcal{A}$ admits a finite presentation 
$$
G_\mathcal{A}=\langle x_1,\dots,x_\ell\mid R_1,\dots,R_s\rangle
$$
where $x_i$ is a meridian of the line~$L_i$ and the words $R_j$ are commutators. Hence $H$
is a free abelian group of rank~$\ell$, generated by the classes $t_1,\dots,t_\ell$ of meridians of
the lines $x_1,\dots,x_\ell$. As a consequence, $\Lambda=\mathbb{Z}[t_1^{\pm1},\dots,t_\ell^{\pm1}]$
is a ring of Laurent polynomials. The augmentation ideal $\mathfrak{m}$ is generated by the polynomials
$\{t_i-1\}_{i=1}^{\ell}$.

The Alexander invariant of $\mathcal{A}$ is generated as a $\Lambda$-module by $x_{i,j}$, $1\leq i<j\leq\ell$, 
the class of $[x_i,x_j]=x_i x_j x_i^{-1} x_j^{-1}$ in $M$. Note that 
$$
[x_i,x_j x_k]=x_{i,j}+t_j x_{i,k} \mod G''.
$$
This way, each relation $R_k$ induces a linear combination~$\tilde{R}_{k}$ of $\{x_{i j}\}$ with coefficients
in $\Lambda$. As was shown in~\cite[Proposition 2.8]{accm:03a}, the module $M(\mathcal{A})$ is the quotient of 
the free module generated by $x_{i,j}$, $1\leq i<j\leq\ell$, by the submodule generated by 
$\tilde{R}_1,\dots,\tilde{R}_{s}$, and the so-called Jacobi relations
\begin{equation}
\label{eq-jacobi}
(t_i-1) x_{j,k}- (t_j-1) x_{i, k}+ (t_k-1) x_{i,j},\quad
1\leq i<j<k\leq\ell.
\end{equation}
\end{ejm}

\begin{obs}
Each multiple point $P$ produces $m(P)-1$ relations. These relations read in~$M_1$ as
follows. If $P$ is the intersection point of $L_{i_1},\dots,L_{i_m}$, then
it produces
\begin{equation}
\forall k\in\{2,\dots,m\}\quad
\sum_{j=1}^m x_{i_k,i_j}\overset{1}{\equiv} 0. 
\end{equation} 
\end{obs}

\begin{ntc}
We will denote $\s_i:=t_i-1$. Note that $\Lambda/\mathfrak{m}^k$ is isomorphic to 
$\mathbb{Z}[\s_1,\dots,\s_\ell]/\n^k$, where $\n$ is the ideal generated by $\s_1,\dots,\s_\ell$. 
The classes in $\Lambda/\mathfrak{m}^k$ are represented by polynomials in $\s_1,\dots,\s_\ell$ of 
degree less than~$k$. Note that the units are those polynomials whose degree zero coefficient equals~$\pm 1$. 
\end{ntc}

% Note that given any finitely presented group $G$, there is a natural map $G'\to M_k$ which factors through
% $G'\to G'/G''$. 
\begin{prop}[{\cite[Proposition 2.15]{accm:03a}}]
\label{propos-modk}
Let $\psi(X_1,...,X_m)$ be a word in the letters $\{X_1,...,X_m\}$. If $p_i,q_i\in G'$ and 
$p_i\overset{k}{\equiv}q_i$ $(i=1,...,m)$, then $[g,\psi(p_1,...,p_m)]\overset{k+1}{\equiv}[g,\psi(q_1,...,q_m)]$, 
$\forall g\in G$. In particular, if $p\in M_k(\mathcal{A})$ then $[g,p]$ is a well-defined element of $M_{k+1}(\mathcal{A})$; 
if $g=x_i$ this element will be written as~$\s_i p\in M_{k+1}(\mathcal{A})$.
\end{prop}

Note that $M_k(\mathcal{A})$ is isomorphic (as an Abelian group) to the graduate 
$\gr^0 M_k(\mathcal{A})\oplus ...\oplus \gr^{k-1} M_k(\mathcal{A})$ by means of the  morphism
$$
p\cdot x_{i,j}\mapsto p_0x_{i,j} + p_1x_{i,j} + ... + p_{k-1}x_{i,j},
$$
where $p$ is a polynomial in $\{\s_1,...,\s_r\}$ and 
$p=p_0+...+p_{k-1}$ is its homogeneous decomposition. 
Note that this isomorphism is not canonical, since it depends on the given set of generators of~$G$. 

For instance, any automorphism of $G$ that sends $x_i$ to $x_i \alpha_i$, 
(with $\alpha_i\in G'$) induces an automorphism of $M_k(\mathcal{A})$:
\begin{equation}
\label{eq-key}
[x_i,x_j]\mapsto [x_i {\alpha_i},x_j {\alpha_j}]\overset{k}{\equiv}
[x_i,x_j]+t_j \s_i\alpha_j-t_i \s_j\alpha_i.
\end{equation}
Since this automorphism respects the filtration it also induces
an automorphism of $\gr M_k(\mathcal{A})$. Note that the automorphism of~$\gr M_k(\mathcal{A})$
it is always the identity but, in general, the automorphism of~$M_k(\mathcal{A})$ is non-trivial.

Let us study the relationship of this invariants with combinatorics. Let us fix
a combinatorics~$\mathscr{C}$ and a realization~$\cA$ of~$\mathscr{C}$.
% In \S\ref{sec-rigid} we defined the free module $H_{\mathscr{C}}^1$; we are
% going to consider its dual $H^{\mathscr{C}}_1$, i.e., 
% the quotient of $\mathbb{Z}^{\ell+1}$ (generated by $x_L$) by the
% submodule generated by $\sum_{L\in\mathcal{L}} x_L$. 
% Recall that the abelian free group~$H_\cA$ is generated also by the
% classes of positive meridians~$g_L$ subject to the relation $\sum_{L\in\mathcal{L}} g_L=0$.
In \S\ref{sec-rigid} we defined the free abelian module $H^{\mathscr{C}}_1$
generated by $x_L$, $L\in\mathcal{L}$, with the relation $\sum_{L\in\mathcal{L}} x_L=0$.
Recall that the abelian free group~$H_\cA$ is generated also by the
classes of positive meridians~$g_L$ subject to the relation $\sum_{L\in\mathcal{L}} g_L=0$
and there is a natural isomorphism between $H^{\mathscr{C}}_1$ and~$H_\cA$.
The composition  
\ifcd
\begin{equation*}
\begin{tikzcd}[column sep=.5em]
G_{\mathcal{A}}\arrow[rrrr]\arrow[drrrr,"h"]&&&&H_{\mathcal{A}}\arrow[d]&g_L\arrow[d, mapsto]\\
&&&&H_1^{\mathscr{C}}&x_L
\end{tikzcd}
\end{equation*}
\else
$h:G_\cA\to H_\cA\to H_1^{\mathscr{C}}$
\begin{equation*}
\begin{tikzpicture}[description/.style={fill=white,inner sep=2pt},baseline=(current bounding box.center)]
\matrix (m) [matrix of math nodes, row sep=2.5em,
column sep=0.5em, text height=1.5ex, text depth=0.25ex]
{G_{\mathcal{A}}&&&&&H_{\mathcal{A}}&g_L\\
{}&&&&&H_1^{\mathscr{C}}&x_L\\};
 \path[->,>=angle 90](m-1-1) edge  (m-1-6);
 \path[->,>=angle 90](m-1-1) edge node[above] {$h$} (m-2-6);
\path[->,>=angle 90](m-1-6) edge(m-2-6);
\path[|->,>=angle 90](m-1-7) edge(m-2-7);
\end{tikzpicture}
\end{equation*}
\fi 
is a \emph{meridian structure} 
of the group $G_\cA$. A meridian $g_{L_i}$ has been denoted $x_i$ above; we will identify
$x_i$ with $x_{L_i}$ without further notice. 

\begin{obs}
Note that given a \emph{meridian structure}~$h$ in $G_\cA$ it is in general
not possible to recover the conjugacy classes of the meridians, only their
homology classes are fixed.
\end{obs}

We are going to define the \emph{naive Alexander invariant} $M^\mathscr{C}$ of the combinatorics~$\mathscr{C}$.
Let $H_1^{\mathscr{C},\Lambda}:=H_1^\mathscr{C}\otimes_{\mathbb{Z}}\Lambda$; then $M^\mathscr{C}$
is defined 
as the quotient of $H_1^{\mathscr{C},\Lambda}\wedge H_1^{\mathscr{C},\Lambda}$, 
with generators $x_{i,j}:=x_i\wedge x_j$, by a submodule generated by
two sets, one coming
coming from the combinatorics
\begin{equation}\label{eq-rels-comb}
\sum_{j=1}^m x_{i_j,i_k},\quad \forall P=\{L_{i_1},\dots,L_{i_m}\},\forall k\in\{2,\dots,m\}, 
\end{equation}
and other one coming from the Jacobi relations:
\begin{equation}
\sigma_{i_1} x_{i_2,i_3}+\sigma_{i_2} x_{i_3,i_1}+\sigma_{i_3} x_{i_1,i_2},\quad 
\forall L_{i_1},L_{i_2},L_{i_3}\in\mathcal{L}.
\end{equation}
From \ref{eq-rels-comb} we can forget the points~$P$ such that $P\in L$ for some particular~$L$.
The truncated naive Alexander invariants $M_k^\mathscr{C}$ and the graduates $\gr_k M^\mathscr{C}$
are defined accordingly. The following result is straightforward.

\begin{prop}
\label{prop-gr-comb}
For a realization $\cA$, the graduate group $\gr^j M=\gr^j M_k$ (if $j<k$) is isomorphic to
$\gr_k M^\mathscr{C}$. In particular, the graduate groups 
depend only on the combinatorics.
\end{prop}

\begin{obs}\label{obs-combi}
Note that the $\Lambda$-action on $\gr_k M^\mathscr{C}$ is reduced to the $\mathbb{Z}$-action.
Moreover $\gr_1 M^\mathscr{C}\equiv M_1$ coincides with 
$\left(H_1^\mathscr{C}\wedge H_1^\mathscr{C}\right)/H_2^\mathscr{C}$, compare \eqref{eq-rels-comb} and 
\eqref{eq-H2}.
\end{obs}

The analog result for $M_k$ is completely false. We end this section with
a result, immediate consequence of Proposition~\ref{propos-modk}, 
which explains why $M_k$ is a more manageable object.

\begin{cor}
\label{lineal}
Formula~{\rm(\ref{eq-key})} only depends on $\varphi_{k-1}(\alpha_i)\in M_{k-1}$.
\end{cor}

\section{Fundamental groups and the Alexander invariant isomorphism test}
\label{sec-fg}

In this section the main Alexander invariant homomorphism test will be computed. In order to do so
we will need to provide with a presentation of the fundamental groups of two realizations of the 
combinatorics $\AG$. This information will allow us to give a presentation for the truncated Alexander
invariants and finally we will prove the failure of the \emph{Alexander invariant isomorphism test}.
The homological rigidity of the combinatorics will allow for the Alexander invariant isomorphism test 
to be very close to an isomorphism test for fundamental groups.

\subsection{Fundamental groups and braided wiring diagrams}
\label{sec-fg1}
\mbox{}

Our purpose is to compare the fundamental groups $G^i$ of the complements $X_i:=\PP^2\setminus \cA^{\xi^i}$,
$i=1,\dots,4$ of the different realizations of the combinatorics $\AG$ as defined in \S~\ref{sec-combinatorics}.
Note that $X_i$ and $X_{5-i}$ are homeomorphic via the conjugation automorphism $\PP^2\to \PP^2$, defined as
$[x:y:z]\mapsto [\bar x:\bar y:\bar z]$. 
Therefore $G^i\cong G^{5-i}$ and thus it only remains to study 
whether or not~$G^1$ and $G^2$ are isomorphic.
If we have fixed the meridian structure
for each $G^i$ (associated to the realizations), the isomorphism
$G^i\to G^{5-i}$ induces $-1_{H_1^{\mathscr{C}}}$ on $H_1^{\mathscr{C}}$. 

Consider the realizations $\cA^{\xi^i}$, $i=1,2$. In order to study the complements $X_1$ and $X_2$ we will 
use an adaptation of the Zariski-Van Kampen method to obtain a presentation of $G^i$. To begin with, one needs
to project $\PP^2$ from a point not on $\cA^{\xi^i}$ and then obtain a system of braids, whose number of strings
is the degree of the arrangement (i.e., the number of lines). The action of these braids on a free group will provide the required set of
relations on the group. Our variation of the method allows for projections from a point on the arrangement, 
namely the point $P_1$, that is $\pi:\PP^2\setminus \{P^i_1\}\to \PP^1$ defined by $\pi(P)=\overline{PP^i_1}$, 
where $\PP^1$ is identified with the space of lines in $\PP^2$ passing through $P_1^i$. This causes a great 
deal of simplification since the braids obtained have as many strings less as the multiplicity of the chosen 
projection point in the arrangement.

In our situation the number of strings drops from 12, in the classical method, to 7. 
Note that $L^i_1,L^i_4,L^i_5,L^i_9$, and $L^i_{12}$ are points in the image of $\pi$. 
Moreover, denote by $T_1,...,T_4$ the lines joining $P^i_1$ and the remaining double points of the arrangement
(that is, points not on any line passing through $P^i_1$), also denote by $X_i:=\PP^2\setminus \cA^{\xi^i}$ the
complement of the arrangement, and finally $\tilde X_i:=X_i\setminus \bigcup_{j=1}^4 T_j$. Then 
$\pi|_{\tilde X_i}$ is a locally trivial fibration whose fiber is $F:=\CC\setminus \{7\ \text{points}\}$.
Technically, if one blows up the point $P^i_1$, then this fibration can be extended to the exceptional divisor,
which becomes a canonical section. This fibration can be understood from the action of the fundamental group of 
the image $\PP^1\setminus \{9\ \text{points}\}$ on the fiber~$F$. This action is called the monodromy of the 
fibration and can be read off its \emph{braided wiring diagram}. 

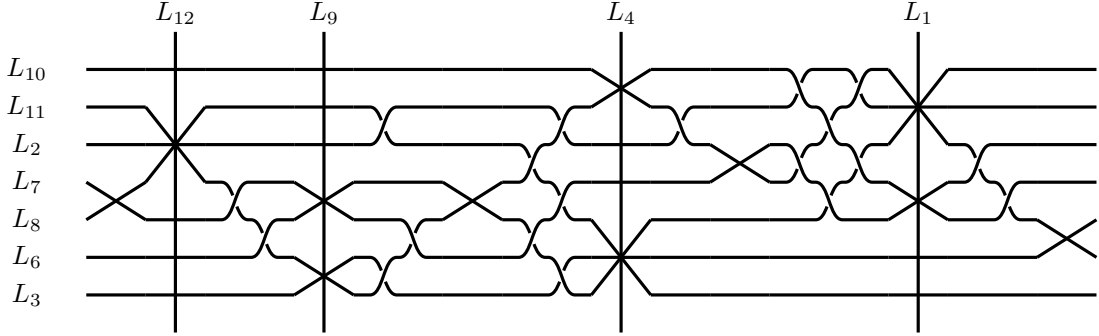
\begin{figure}[ht]
\begin{tikzpicture}[xscale=.79,yscale=.5,vertice/.style={draw,circle,fill,minimum size=0.25cm,inner sep=0}]
\node at (-1,-0) {$L_{10}$};
\node at (-1,-1) {$L_{11}$};
\node at (-1,-2) {$L_{2}$};
\node at (-1,-3) {$L_{7}$};
\node at (-1,-4) {$L_{8}$};
\node at (-1,-5) {$L_{6}$};
\node at (-1,-6) {$L_{3}$};
\begin{scope}[xshift=0cm]
\foreach \x in {0,...,2}
\draw[line width=1.2] (0,-\x)--(1,-\x);\foreach \x in {3,...,4}
\def\y{\x-7}
\draw[line width=1.2] (0,-\x)--(1,\y);\foreach \x in {5,...,6}
\draw[line width=1.2] (0,-\x)--(1,-\x);;
 \end{scope}
\begin{scope}[xshift=1cm]
\draw[line width=1.2] (0.5,1)--(0.5,-7);
\node at (.5,1.5) {$L_{12}$};
\foreach \x in {0,...,0}
\draw[line width=1.2] (0,-\x)--(1,-\x);\foreach \x in {1,...,3}
\def\y{\x-4}
\draw[line width=1.2] (0,-\x)--(1,\y);\foreach \x in {4,...,6}
\draw[line width=1.2] (0,-\x)--(1,-\x);;
 \end{scope}
\begin{scope}[xshift=2cm]
\braid[line width=1.2,xscale=-1,rotate= -90,height=.5cm,width=1cm,number of strands=7] 
at (0,0)
s_4^{-1} s_5^{-1} ;;
 \end{scope}
\begin{scope}[xshift=3.5cm]
\draw[line width=1.2] (0.5,1)--(0.5,-7);
\node at (.5,1.5) {$L_{9}$};
\foreach \x in {0,...,2}
\draw[line width=1.2] (0,-\x)--(1,-\x);\foreach \x in {3,...,4}
\def\y{\x-7}
\draw[line width=1.2] (0,-\x)--(1,\y);\foreach \x in {5,...,6}
\def\y{\x-11}
\draw[line width=1.2] (0,-\x)--(1,\y);;
 \end{scope}
\begin{scope}[xshift=4.5cm]
\braid[line width=1.2,xscale=-1,rotate= -90,height=.5cm,width=1cm,number of strands=7] 
at (0,0)
s_6^{-1}-s_2^{-1} s_5 ;;
 \end{scope}
\begin{scope}[xshift=6cm]
\foreach \x in {0,...,2}
\draw[line width=1.2] (0,-\x)--(1,-\x);\foreach \x in {3,...,4}
\def\y{\x-7}
\draw[line width=1.2] (0,-\x)--(1,\y);\foreach \x in {5,...,6}
\draw[line width=1.2] (0,-\x)--(1,-\x);;
 \end{scope}
\begin{scope}[xshift=7cm]
\braid[line width=1.2,xscale=-1,rotate= -90,height=.5cm,width=1cm,number of strands=7] 
at (0,0)
s_5^{-1}-s_3^{-1} s_4^{-1}-s_2^{-1}-s_6 ;;
 \end{scope}
\begin{scope}[xshift=8.5cm]
\draw[line width=1.2] (0.5,1)--(0.5,-7);
\node at (.5,1.5) {$L_{4}$};
\foreach \x in {0,...,1}
\def\y{\x-1}
\draw[line width=1.2] (0,-\x)--(1,\y);\foreach \x in {2,...,3}
\draw[line width=1.2] (0,-\x)--(1,-\x);\foreach \x in {4,...,6}
\def\y{\x-10}
\draw[line width=1.2] (0,-\x)--(1,\y);;
 \end{scope}
\begin{scope}[xshift=9.5cm]
\braid[line width=1.2,xscale=-1,rotate= -90,height=.5cm,width=1cm,number of strands=7] 
at (0,0)
s_2 ;;
 \end{scope}
\begin{scope}[xshift=10.5cm]
\foreach \x in {0,...,1}
\draw[line width=1.2] (0,-\x)--(1,-\x);\foreach \x in {2,...,3}
\def\y{\x-5}
\draw[line width=1.2] (0,-\x)--(1,\y);\foreach \x in {4,...,6}
\draw[line width=1.2] (0,-\x)--(1,-\x);;
 \end{scope}
\begin{scope}[xshift=11.5cm]
\braid[line width=1.2,xscale=-1,rotate= -90,height=.5cm,width=1cm,number of strands=7] 
at (0,0)
s_3 -s_1 s_4- s_2 s_3 -s_1^{-1} ;;
 \end{scope}
\begin{scope}[xshift=13.5cm]
\draw[line width=1.2] (0.5,1)--(0.5,-7);
\node at (.5,1.5) {$L_{1}$};
\foreach \x in {0,...,2}
\def\y{\x-2}
\draw[line width=1.2] (0,-\x)--(1,\y);\foreach \x in {3,...,4}
\def\y{\x-7}
\draw[line width=1.2] (0,-\x)--(1,\y);\foreach \x in {5,...,6}
\draw[line width=1.2] (0,-\x)--(1,-\x);;
 \end{scope}
\begin{scope}[xshift=14.5cm]
\braid[line width=1.2,xscale=-1,rotate= -90,height=.5cm,width=1cm,number of strands=7] 
at (0,0)
s_3^{-1} s_4^{-1} ;;
 \end{scope}
\begin{scope}[xshift=16cm]
\foreach \x in {0,...,3}
\draw[line width=1.2] (0,-\x)--(1,-\x);\foreach \x in {4,...,5}
\def\y{\x-9}
\draw[line width=1.2] (0,-\x)--(1,\y);\foreach \x in {6,...,6}
\draw[line width=1.2] (0,-\x)--(1,-\x);;
 \end{scope}
\end{tikzpicture}
\caption{Wiring diagram of the arrangement $\mathcal{A}^{\zeta_5}$}
\label{fig:wiring1}
\end{figure}

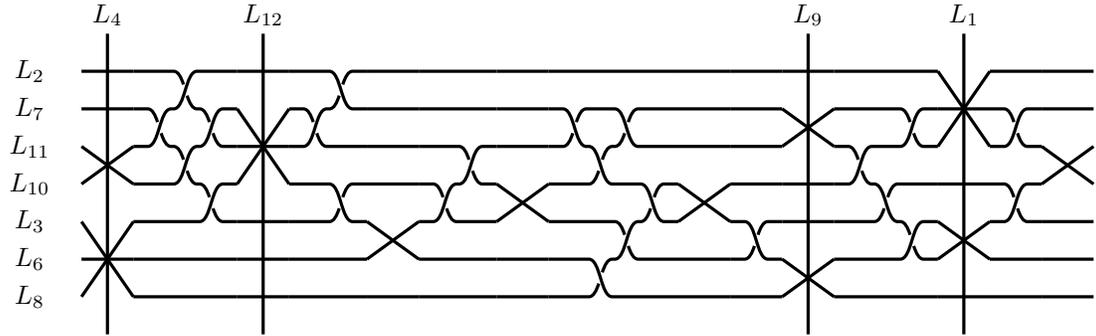
\begin{figure}[ht]
\begin{tikzpicture}[xscale=.69,yscale=.5,vertice/.style={draw,circle,fill,minimum size=0.25cm,inner sep=0}]
\node at (-1,-0) {$L_{2}$};
\node at (-1,-1) {$L_{7}$};
\node at (-1,-2) {$L_{11}$};
\node at (-1,-3) {$L_{10}$};
\node at (-1,-4) {$L_{3}$};
\node at (-1,-5) {$L_{6}$};
\node at (-1,-6) {$L_{8}$};
\begin{scope}[xshift=0cm]
\draw[line width=1.2] (0.5,1)--(0.5,-7);
\node at (.5,1.5) {$L_{4}$};
\foreach \x in {0,...,1}
\draw[line width=1.2] (0,-\x)--(1,-\x);\foreach \x in {2,...,3}
\def\y{\x-5}
\draw[line width=1.2] (0,-\x)--(1,\y);\foreach \x in {4,...,6}
\def\y{\x-10}
\draw[line width=1.2] (0,-\x)--(1,\y);;
\end{scope}
\begin{scope}[xshift=1cm]
\braid[line width=1.2,xscale=-1,rotate= -90,height=.5cm,width=1cm,number
of strands=7]  at (0,0)
s_2^{-1} s_3^{-1} -s_1^{-1} s_4^{-1}- s_2^{-1} ;;
 \end{scope}
\begin{scope}[xshift=3cm]
\draw[line width=1.2] (0.5,1)--(0.5,-7);
\node at (.5,1.5) {$L_{12}$};
\foreach \x in {0,...,0}
\draw[line width=1.2] (0,-\x)--(1,-\x);\foreach \x in {1,...,3}
\def\y{\x-4}
\draw[line width=1.2] (0,-\x)--(1,\y);\foreach \x in {4,...,6}
\draw[line width=1.2] (0,-\x)--(1,-\x);;
 \end{scope}
\begin{scope}[xshift=4cm]
\braid[line width=1.2,xscale=-1,rotate= -90,height=.5cm,width=1cm,number
of strands=7]  at (0,0)
s_2^{-1} s_1 -s_4 ;;
 \end{scope}
\begin{scope}[xshift=5.5cm]
\foreach \x in {0,...,3}
\draw[line width=1.2] (0,-\x)--(1,-\x);\foreach \x in {4,...,5}
\def\y{\x-9}
\draw[line width=1.2] (0,-\x)--(1,\y);\foreach \x in {6,...,6}
\draw[line width=1.2] (0,-\x)--(1,-\x);;
 \end{scope}
\begin{scope}[xshift=6.5cm]
\braid[line width=1.2,xscale=-1,rotate= -90,height=.5cm,width=1cm,number
of strands=7]  at (0,0)
s_4^{-1} s_3^{-1} ;;
 \end{scope}
\begin{scope}[xshift=8cm]
\foreach \x in {0,...,2}
\draw[line width=1.2] (0,-\x)--(1,-\x);\foreach \x in {3,...,4}
\def\y{\x-7}
\draw[line width=1.2] (0,-\x)--(1,\y);\foreach \x in {5,...,6}
\draw[line width=1.2] (0,-\x)--(1,-\x);;
 \end{scope}\begin{scope}[xshift=9cm]
\braid[line width=1.2,xscale=-1,rotate= -90,height=.5cm,width=1cm,number
of strands=7]  at (0,0)
s_2 s_3 -s_6 s_2- s_5^{-1} s_4 ;;
 \end{scope}
\begin{scope}[xshift=11.5cm]
\foreach \x in {0,...,2}
\draw[line width=1.2] (0,-\x)--(1,-\x);\foreach \x in {3,...,4}
\def\y{\x-7}
\draw[line width=1.2] (0,-\x)--(1,\y);\foreach \x in {5,...,6}
\draw[line width=1.2] (0,-\x)--(1,-\x);;
 \end{scope}
\begin{scope}[xshift=12.5cm]
\braid[line width=1.2,xscale=-1,rotate= -90,height=.5cm,width=1cm,number
of strands=7]  at (0,0)
s_5 ;;
 \end{scope}
\begin{scope}[xshift=13.5cm]
\draw[line width=1.2] (0.5,1)--(0.5,-7);
\node at (.5,1.5) {$L_{9}$};
\foreach \x in {0,...,0}
\draw[line width=1.2] (0,-\x)--(1,-\x);\foreach \x in {1,...,2}
\def\y{\x-3}
\draw[line width=1.2] (0,-\x)--(1,\y);\foreach \x in {3,...,4}
\draw[line width=1.2] (0,-\x)--(1,-\x);\foreach \x in {5,...,6}
\def\y{\x-11}
\draw[line width=1.2] (0,-\x)--(1,\y);;
 \end{scope}
\begin{scope}[xshift=14.5cm]
\braid[line width=1.2,xscale=-1,rotate= -90,height=.5cm,width=1cm,number
of strands=7]  at (0,0)
s_3^{-1} s_4 s_5^{-1} -s_2^{-1} ;;
 \end{scope}
\begin{scope}[xshift=16.5cm]
\draw[line width=1.2] (0.5,1)--(0.5,-7);
\node at (.5,1.5) {$L_{1}$};
\foreach \x in {0,...,2}
\def\y{\x-2}
\draw[line width=1.2] (0,-\x)--(1,\y);\foreach \x in {3,...,3}
\draw[line width=1.2] (0,-\x)--(1,-\x);\foreach \x in {4,...,5}
\def\y{\x-9}
\draw[line width=1.2] (0,-\x)--(1,\y);\foreach \x in {6,...,6}
\draw[line width=1.2] (0,-\x)--(1,-\x);;
 \end{scope}
\begin{scope}[xshift=17.5cm]
\braid[line width=1.2,xscale=-1,rotate= -90,height=.5cm,width=1cm,number
of strands=7]  at (0,0)
s_2^{-1} -s_4^{-1} ;;
 \end{scope}
\begin{scope}[xshift=18.5cm]
\foreach \x in {0,...,1}
\draw[line width=1.2] (0,-\x)--(1,-\x);\foreach \x in {2,...,3}
\def\y{\x-5}
\draw[line width=1.2] (0,-\x)--(1,\y);\foreach \x in {4,...,6}
\draw[line width=1.2] (0,-\x)--(1,-\x);;
 \end{scope}
 \end{tikzpicture}
\caption{Wiring diagram of the arrangement $\mathcal{A}^{\zeta_5^2}$}
\label{fig:wiring2}
\end{figure}

In a nutshell, a braided wiring diagram describes the preimage in $\cA^{\xi^i}$ by $\pi$ of a closed path 
$\gamma$ on the base $\PP^1$ starting at a base point and going through all the points in the discriminant of~$\pi$. 
This will be represented by a degenerated braid with crossings and multiple points (one multiple point for each singular point of the arrangement except for 
$P^i_1$). For more details, see~\cite{Cohen-Suciu-braid}. To the original definition of braided wiring diagram we
have added the possibility of projecting from a singular point of the arrangement. This results in the vertical 
lines shown in the figures, corresponding to the degenerated fibers in the arrangement which have to be plotted too.
Moreover, since the loop surrounding all the points on the discriminant is trivial in the fundamental group 
of the base, the closed path $\gamma$ could skip one point on the discriminant and still contain all the 
necessary information to recover the fibration.

Braided wiring diagrams for $\cA^{\xi^i}$, $i=1,2$ as described above for paths skipping the image of the 
line $L^i_5$ are shown in Figures~\ref{fig:wiring1} and~\ref{fig:wiring2}. These plots can be obtained 
in a precise way, using the fact that lines can be parametrized.

% From these diagrams one obtains immediately the monodromy presentation of the fundamental group 
% $G^i:=\pi_1(X_i)$ as follows.
% Let $I$ describe the set of indices of non-vertical lines of the projection $\pi$ and let $\{\mu_i\}_{i\in I}$ 
% be the meridians generating the fundamental group of the fiber (in our case 
% $\FF_7=\pi_1(\CC\setminus \{7\ \text{points}\})$). On the other hand one has meridians $\{\mu_j\}_{j\in J}$ 
% for the vertical lines on, say, the canonical section. 
% The fundamental group $G^i$ is generated by $\{\mu_k\}_{k\in I\cup J}$.
% For each point $Q$ on the discriminant of the projection $\pi$ we will obtain a number of relations.

In order to simplify the way of computing the fundamental groups,
we perturb a little bit the projection point inside the line at infinity to obtain a wiring diagram
without vertical lines,
see~\cite{bgb:14a}. We obtain
a wiring diagram of~$11$ lines, where the former 
vertical lines start up in the left-hand side (the rightest the highest) and end below
in the right-hand side. From these diagrams one obtains immediately the braid
monodromy which allows to compute the fundamental group 
$G^i:=\pi_1(X_i)$ as follows.
Let $\mu_i$ be a geometric basis of meridians  generating the fundamental group of the fiber (in our case 
$\FF_{11}=\pi_1(\CC\setminus \{11\ \text{points}\})$).
The fundamental group $G^i$ is generated by $\{\mu_k\}_{k=1}^{11}$.

Let us explain how each point $Q$ on the discriminant of the projection $\pi$ induces
a number of relations. The point~$Q$ is associated
to a multiple point  involving  $m_Q$ strands with indices
$i_Q+1,\dots,i_Q+m_Q$ (these indices correspond to the order of the strands near~$Q$, not to
the labeling of the lines). Turning around the point~$Q$
is associated to the braid $\Delta_Q^2$, which consists on the full-twist
of those $m$~strands, and straight lines for the other strands.
If $\mu_1^Q,\dots,\mu_{11}^Q$ is a geometric basis of meridians in a vertical line close to~$x=Q$,
the relations obtained are:
$$
\mu_j^Q=(\mu_j^Q)^{\Delta_Q^2},\quad 1\leq j\leq 11.
$$
The exponent stands for the geometric action of the braid group on the free group.
Any one of above the relations is a consequence of the other ones, but we can be more precise.
If $\mu(Q)=\prod_{j=1}^{m_Q}\mu_j$, the above relations read as
\begin{equation}
 \mu_j^Q=
\begin{cases}
\mu(Q)^{-1}\cdot\mu_j\cdot\mu(Q)&\text{ if }j=i_Q+1,\dots,i_Q+m_Q,\\
\mu_j&\text{ otherwise.}        
\end{cases}
\end{equation}
Eliminating unnecessary relations, we keep:
\begin{equation}
\mu_{i_Q+j}^Q=
\mu(Q)^{-1}\cdot\mu_j\cdot\mu(Q)\quad \forall
j=1,\dots,m_Q-1.       
\end{equation}
The wiring diagram provides the global information which allows 
to deal with all these relations together. Namely there is a braid $\beta_Q$ connecting 
the generic vertical
line close to~$Q$ and a generic vertical line in left-hand side of the diagram. 
The global braid is $\alpha_Q:=\beta_Q^{-1}\cdot\Delta_Q^2\cdot\beta_Q$. 
This braid produces the relations
\begin{equation}
\label{eq-rel1}
\mu_i=\mu_i^{\alpha_Q},\quad 1\leq i\leq 11.
\end{equation}
As before, we can reduce these relations as:
\begin{equation}
\label{eq-rel2}
\left[\mu_{i_Q+j}^{\beta_Q},\mu(Q)^{\beta_Q}\right]=1,\quad j=1,\dots,m_Q-1.
\end{equation}
One has the following.

\begin{prop}
\label{prop-zariski}
The fundamental group $G^i$ admits a presentation
$$
\langle \mu_k, k\in\{1,\dots,11\}\mid \eqref{eq-rel2}\rangle.
$$
\end{prop}

Using Proposition~\ref{prop-zariski} and the braided wiring diagrams given above the presentations
can be obtained. They are coded in~\autoref{prt1}. The correspondence between the generators
$\mu_k$, $k=1,\dots,11$, used in~\autoref{prt1}, and the numeration of the twelve lines 
of~\S\ref{sec-combinatorics} is given in \autoref{table:lines}.

\begin{table}[ht]
\begin{center}
\begin{tabular}{|l|l|l|l|l|l|l|l|l|l|l|l|l|}\hline
Line &  $L_7$& $L_8$ &$L_{12}$& $L_{10}$ & $L_{11}$ & $L_2$ & $L_6$ & $L_3$ & $L_9$ & $L_4$ & $L_1$ & $L_5$  \\\hline
Meridian & $\mu_{10}$ & $\mu_{9}$ & $\mu_{5}$ & $\mu_{8}$ & $\mu_{11}$ & $\mu_{4}$ &$\mu_{6}$  
& $\mu_{2}$ & $\mu_{7}$ & $\mu_{3}$ & $\mu_{1}$ & \\\hline
\end{tabular}
\end{center}
\caption{Numeration of lines and meridians}
\label{table:lines}
\end{table}

\subsection{Alexander invariant computations and the AI-isomorphism test}
\label{sec-fg2}
\mbox{}

To avoid the annoying problem that we consider no meridian for the line~$L_5$, from now on
we number the lines with the index of their meridians; the \emph{line at infinity} (former line~$L_5$)
is now the line~$L_0$.

From the relations induced by \eqref{eq-rel2} and the Jacobi relations~\eqref{eq-jacobi} 
we obtain a presentation for the 
truncated Alexander invariants $M^i_k$, $k=1,2$ (see Definition~\ref{dfn-truncated} and Example~\ref{ejem-truncated}).
Note that $M^i_1$ depends only on the combinatorics $\AG$. In particular, in our case, $M^i_1$ is generated by
$\{x_{j,k}:=x_{L_j,L_k}\mid 1\leq j,k\leq 11\}$
and relations $\sum_{L_k<P} x_{j,k}=0$ for each $L_j<P$, and for each~$P$ such that $L_0\not< P$. We assume the convention
that $x_{k,j}=-x_{j,k}$.

This implies that $M^i_1\cong H_2$ via the morphism $x_{j,k}\mapsto x_{L_j}\wedge x_{L_k}$ as a consequence of~\eqref{eq-H2}.
We can choose
$\mathcal{B}\subset\{(j,k)\mid 1\leq j<k\leq 11\}$ such that $\{x_{j,k}\mid (j,k)\in\mathcal{B}\}$
is a basis of $M^i_1$.

As for $M^i_2$, the actual presentations of $G^i$ and the Jacobi relations~\eqref{eq-jacobi} are in this case required. 
Note that  $\rank M^i_2=\rank\gr^0 M^i_2+\rank\gr^1 M^i_2$. With the relations
coming from the group we can see that $\{\sigma_h\cdot x_{j,k}\mid 1\leq h\leq 11,(j,k)\in\mathcal{B}\}$
generate $M^i_2$, subject to the relations induced by~\eqref{eq-jacobi}. 

Let us describe the AI-isomorphism test in a general setting. 
Assume $(G^1,h_1)$ and $(G^2,h_2)$ are fundamental groups of two 
arrangements with the same combinatorics $\mathscr C=(\cL,\cP)$ and with their meridian structures.
Let $\ell=|\mathcal{L}|-1$ and let $L_0,L_1,\dots,L_\ell$ be the lines
and choose a set~$\mathcal{B}$ as above.
Assume there is an isomorphism~$\varphi$ fitting in the following diagram:
\ifcd
\begin{equation}
\begin{tikzcd}[row sep=normal]
G^1\arrow[rr,"\varphi"]\arrow[rd,"h_1" below left=3pt]&&G^2\arrow[dl,"h_2"]\\
&H_1^{\mathscr{C}}&
\end{tikzcd}
\end{equation}
\else
\begin{equation}
\begin{tikzpicture}[description/.style={fill=white,inner sep=2pt},baseline=(current bounding box.center)]
\matrix (m) [matrix of math nodes, row sep=2.5em,
column sep=2.5em, text height=1.5ex, text depth=0.25ex]
{G^1&&G^2\\
&H_1^{\mathscr{C}}&\\};
\path[->,>=angle 90](m-1-1) edge node[above] {$\varphi$} (m-1-3);
\path[->,>=angle 90](m-1-1) edge node[below] {$h_1$} (m-2-2);
\path[->,>=angle 90](m-1-3) edge node[below right] {$h_2$} (m-2-2);
\end{tikzpicture}
\end{equation}
\fi
For $1\leq k\leq\ell$, we  write $\mu_{k,i}$ for 
a generator of $G^i$ which is a meridian whose image by $h_i$ is $x_k:=x_{L_k}\in H_1^{\mathscr C}$.
Note that $G^i$ is generated by $\mu_{k,i}$, $1\leq k\leq\ell$.

\begin{enumerate}
\enet{(AI-\arabic{enumi})}
\item The isomorphism $\varphi$ is determined by 
$\varphi(\mu_{k,1})= \mu_{k,2} \cdot g_k$, where $g_k\in(G^2)'$.
Note that this does not necessarily mean that $\varphi$ preserves meridians, since $\mu_{k,1} \cdot g_k$ is
not necessarily in the same conjugacy class as $\mu_{k,2}$.

\item The isomorphism $\varphi$  induces a $\Lambda/\mathfrak{m}_2$-isomorphism $\varphi_*:M_2^1\to M_2^2$.
Let us denote $x_{j,k}^i:=[\mu_{j,i},\mu_{k,i}]\bmod \Lambda/\mathfrak{m}^2$. If no confusion is likely to arise,
we will drop the super-index. We have
\begin{equation}\label{eq:hom}
x_{j,k}\mapsto [\mu_{j,2}\cdot g_j,\mu_{k,2}\cdot g_{k}]\bmod \Lambda/\mathfrak{m}^2. 
\end{equation}

\item Because of \autoref{lineal}, only the class of $g_j$ in $M_1^2$ is needed, 
$\varphi_*$ is determined by 
\begin{equation}\label{eq:rel1}
g_j\overset{1}{\equiv}\sum_{(h,k)\in\mathcal{B}} n_{j,h,k}\cdot x_{h,k},
\end{equation}
for some $n_{j,h,k}\in\mathbb{Z}$.

\item Take a relation $R\in G^1$. This relation induces a linear equality $R_1$ in 
$M_2^1$, written in terms of $x_{i,j}\in M_2^1$ for $1\leq i< j\leq \ell$.
Using~\eqref{eq:hom} and~\eqref{eq:rel1}, its image under $\varphi_*$
is an equation $R_2$ in the module $M_2^2$, whose unknowns are the variables
$n_{j,h,k}$ of \eqref{eq:rel1}.

\item One can reduce the equation $R_2$ using the relations of $M_2^2$
to obtain equations in $\gr^1 M_2$.
\end{enumerate}

The existence of solutions for this system  is a necessary condition for the  
existence of such an isomorphism $\varphi$. 
The process we have outlined above is referred to as the \emph{AI-isomorphism test} of level~$2$.
If the necessary condition is true, we can use the integer solutions to work in~$M_3$, and so on.
This process at each $M_k$ will be referred to as the \emph{AI-isomorphism test} of level~$k$ for
the pair~$[(G^1,h_1),(G^2,h_2)]$.

\begin{prop}
\label{prop-AI1}
Under the notation above, the pair $[(G^1,h_1),(G^2,h_2)]$  do not pass the 
AI-isomorphism test of level~$2$.
In other words, there is no isomorphism $(G^1,h_1)\to (G^2,h_2)$ inducing 
the identity~$1_{H_1^{\AG}}$.
\end{prop}

\begin{proof}
One can easily check that $\rank M^i_1=23$. A long computation gives
$\rank\gr^1 M^i=91$, hence each equation $\tilde{R}_2$ in $\gr^1 M^2$ 
induces~$91$ linear equations;
since there are~$32$ such relations, we obtain a linear system of
$2912$ equations in~$23\times 11=253$ unknowns. 
After eliminating trivial equations, we have to deal with a linear system of $930$ equations
with $253$ variables. We have attacked this system with~\texttt{Sagemath}~\cite{sage67},
using pivoting methods where divisions are not allowed. 

We obtain that
the solutions over~$\mathbb{Q}$ is an affine space of dimension~$12$,
but the smallest ring where solutions exist is $\mathbb{Z}\left[\frac{1}{5}\right]$.
In particular, no integer solution exists. These computations, together with the presentation of $M^i_1$ 
and $M^i_2$  took 662.49s of CPU time on an 8 double core Athlon 
processor computer with 128GB of RAM.
\end{proof}

We can apply the same ideas to the groups $G^4$ and $G^1$.
Let us recall that, even though $G^1$ and $G^4$ are isomorphic, their meridian
structures are reversed, i.e. the known isomorphism induces $-1_{H_1^{\AG}}$.

\begin{prop}
\label{prop-AI2}
Under the notation above, the pair $[(G^4,h_4),(G^2,h_2)]$  do not pass the 
AI-isomorphism test of level~$2$.
In other words, there is no isomorphism $(G^4,h_4)\to (G^2,h_2)$ inducing 
the identity~$1_{H_1^{\AG}}$.
\end{prop}

\begin{proof}
We repeat the above procedure skipping the computations about $M_2^2$, already done. After
97.69s of  CPU time, we obtain similar results as in \autoref{prop-AI2}.
\end{proof}

With all the above we are ready to prove the main result.

\begin{thm}
\label{thm-isom} 
The fundamental groups $G^1$ and $G^2$ of the complement to the realizations $\cA^{\xi}$ and $\cA^{\xi^2}$ 
of the combinatorics $\AG$ are not isomorphic.
\end{thm}

\begin{proof}
Consider the presentations of $G^i$ given in \S\ref{sec-combinatorics}, which provide meridian structures
in a natural way. Assume $\varphi:G^1\to G^2$ is an arbitrary isomorphism. By \autoref{prop-hom-rig}, $\AG$ is 
homologically rigid and $\Aut_\AG=1_{\AG}$. As a consequence, $\varphi$ induces $\varphi_*=\pm 1_{H_1^{\AG}}$,
see~\autoref{obs-admis}. 
If $\varphi_*=1_{H_1^{\AG}}$, then \autoref{prop-AI1} results in a contradiction.
Assume $\varphi_*=-1_{H_1^{\AG}}$. Note that, as mentioned at the beginning 
of~\eqref{sec-fg1}, the conjugation morphism $\psi$ induces an automorphism of groups whose associated 
morphism $\psi_*$ equals $-1_{H^1_{\AG}}$. Composing $\psi\circ \varphi$ one obtains an isomorphism 
$\varphi\circ \psi :G^4\to G^2$ satisfying $(\varphi\circ \psi)_*=1_{H^1_{\AG}}$, 
which can be disregarded using \autoref{prop-AI2}. This ends the proof.
\end{proof}

% \bibliographystyle{amsplain}
% \bibliography{biblio_paper}
\providecommand{\bysame}{\leavevmode\hbox to3em{\hrulefill}\thinspace}
\providecommand{\MR}{\relax\ifhmode\unskip\space\fi MR }
\providecommand{\MRhref}[2]{%
  \href{http://www.ams.org/mathscinet-getitem?mr=#1}{#2}
}
\providecommand{\href}[2]{#2}

\appendix
\section{Sagemath code}\label{sec-code}
\begin{prt}\label{prt1}
Coding the wiring diagrams
\end{prt}
The following code presents the two wiring diagrams, and the groups $\mathbb{B}_n,\mathbb{F}_n$,
for $n=11$. For further use, we introduce the permutations \texttt{orden$\pm$} relating the order at
the base vertical line of the wiring diagram. The lists \texttt{wiring$\pm$} have one element
for each crossing. Each element contains two parts: the previous braid and the lines involved
in the crossing.

\begin{lstlisting}[breaklines=true]
n=11
B=BraidGroup(n)
F=FreeGroup(n)

ordenPos=Permutation([1,3,7,5,8,11,4,10,9,6,2])
wiringPos=[[(), [10, 9]], [(), [5, 8]], [(), [5, 11, 4, 9]], [(), [5, 10]], 
[(),[5, 6]], [(), [5, 2]], [(-7, -8), [7, 8]], [(), [7, 9]], [(), [7, 4]],
[(), [7, 10, 6]], [(), [7, 11, 2]], [(-8, -4, 7), [6, 11]], 
[(-7, -5,-6, -4, 8), [3, 8, 11]], [(), [3, 4]], [(), [3, 10]], 
[(), [3, 9, 2,6]], [(3,), [8, 10]], [(4, 5, 2, 3, 4, -2), [1, 8, 4, 6]], 
[(), [1, 11,10]], [(), [1, 2]], [(), [1, 9]], [(-3, -4), [8, 2]]]

ordenNeg=Permutation([1,7,5,3,4,10,11,8,2,6,9])
wiringNeg=[[(), [3, 4]], [(), [3, 10]], [(), [3, 11, 8]], [(), [3, 2, 6, 9]],
[(-5, -6, -7, -4, -5), [5, 8]], [(), [5, 11, 4, 9]], [(), [5, 10]], 
[(),[5, 6]], [(), [5, 2]], [(-4, 3, 6), [11, 6]], [(-6, -5), [9, 10]],
 [(4,5, 4, 8, -7, 6), [2, 8]], [(7,), [7, 4]], [(), [7, 10, 6]], 
[(), [7,8]], [(), [7, 9]], [(), [7, 2, 11]], [(-4, 5, -6, -3), [1, 4, 8, 6]],
[(), [1, 9]], [(), [1, 11, 10]], [(), [1, 2]], [(-2, -4), [8, 10]]]
\end{lstlisting}

\begin{prt}
Core of the computation
\end{prt}

\begin{paso}\label{paso1}
We start constructing the presentations of the groups using the function in~\ref{S-relations}.
We construct the list \texttt{inc} representing a family 
$\mathcal{B}\subset\{(i,j)\mid 1\leq i<n, i< j\leq n\}$ such that
$\{x_{i,j}\mid (i,j)\in\mathcal{B}\}$ is a basis of the truncated Alexander invariant~$M_1$. 
\end{paso}
\begin{lstlisting}[breaklines=true]
GrupoListaPos=list_of_relations_2(ordenPos,wiringPos)
GrupoListaNeg=list_of_relations_2(ordenNeg,wiringNeg)
combPos=[sorted(v[0]) for v in GrupoListaPos]
combNeg=[sorted(v[0]) for v in GrupoListaNeg]
print "Do combinatorics coincide? ",sorted(combPos)==sorted(combNeg)
comb=[]
for v in combPos:
    for j in v[1:]:
        comb.append([v[0],j])
tot=[tuple(sorted(_.list())) for _ in Subsets([1..n],2)] 
inc=[_ for _ in tot if list(_) not in comb]
\end{lstlisting}

\begin{paso}
We construct several rings. We start with $S:=\mathbb{Z}[x_{k,i,j}\mid 1\leq k\leq n,(i,j)\in\mathcal{B}]$.
The variables of this ring are the unknowns of the final linear system. We define
a dictionnary \texttt{YY} to track the variables of~$S$ using the the triples $(k,i,j)$.
The ring \texttt{LR} is the Laurent polynomial ring $\mathbb{Z}[t_1^{\pm 1},\dots,t_n^{\pm 1}]=\mathbb{Z}[H]$
and \texttt{LRv} is  $S[t_1^{\pm 1}\dots,t_n^{\pm 1}]=\mathbb{Z}[H]$ while
$R:=\mathbb{Z}[t_1,\dots,t_n]$, $T:=S[[\sigma_1,\dots,\sigma_n]]$ and $T_0:=\mathbb{Z}[[\sigma_1,\dots,\sigma_n]]$.
The dictionnaries \texttt{dic} and \texttt{dicv} realize the substitution $t_i\mapsto 1+\sigma_i$
in each ring.
\end{paso}

\begin{lstlisting}[breaklines=true]
xx=var(['v%da%db%d'% (k,i,j) for k in [1..n] for (i,j) in inc])
S=PolynomialRing(ZZ,xx)
YY={}
for k in [1..n]:
    for l in range(len(inc)):
        YY[k,inc[l][0],inc[l][1]]=S.gen((k-1)*len(inc)+l)
LR=LaurentPolynomialRing(ZZ,'t',n)
LRv=LR.change_ring(S)
R=LR.polynomial_ring()
T=PowerSeriesRing(S,'s',num_gens=n,default_prec=2)
T0=T.change_ring(ZZ)
HomLRT0=R.hom([1+v for v in T0.gens()],codomain=T0,check=True)
dic={v:HomLRT0(v) for v in LR.gens()}
dicv={v:HomLRT0(v).change_ring(S) for v in LRv.gens()} 
\end{lstlisting}

\begin{paso}
We create the free module $M$ with basis $x_{i,j}$ and base ring
\texttt{LR}. The dictionnary \texttt{XX} expresses any $[x_i^{\pm 1},x_j^{\pm 1}]$ in the module. 
\end{paso}

\begin{lstlisting}[breaklines=true]
M=FreeModule(LR,n*(n-1)/2)
MT=M.change_ring(T0)
Mv=M.change_ring(T)
XX={tot[i]:M.gen(i) for i in range(n*(n-1)/2)}
for i in range(1,n+1):
    XX[i,i]=M(0)
    XX[-i,-i]=M(0)
    XX[i,-i]=M(0)
    XX[-i,i]=M(0)
    for j in range(i+1,n+1):
        XX[-i,j]=-tt([i],LR)^-1*XX[i,j]
        XX[i,-j]=-tt([j],LR)^-1*XX[i,j]
        XX[-i,-j]=tt([i],LR)^-1*tt([j],LR)^-1*XX[i,j]
        XX[j,i]=-XX[i,j]
        XX[-j,i]=-XX[i,-j]
        XX[j,-i]=-XX[-i,j]
        XX[-j,-i]=-XX[-i,-j] 
\end{lstlisting}

\begin{paso}
We translate the relations of both groups as elements of the module~$M$. We
write down both sets of relations as matrices (the number of rows
is the number of relations and columns related to all the $x_{i,j}$
and we translate them into the series ring.
The matrix \texttt{SustNeg} expresses all the
$x_{i,j}$ in terms of those such that $(i,j)\in\mathcal{B}$.
\end{paso}

\begin{lstlisting}[breaklines=true]
relAlexPos=[]
for v in GrupoListaPos:
    relAlexPos+=CommCyclic1(v,F,LR,XX,dir='LR')
print "Alexander invariant for first group"
relAlexNeg=[]
for v in GrupoListaNeg:
    relAlexNeg+=CommCyclic1(v,F,LR,XX,dir='LR')
nrels=len(relAlexNeg)
print "Alexander invariant for the second group"
relAlexSeriePos=Matrix([vector(w.subs(dic) for w in v.list()) for v in relAlexPos])
relAlexSerieNeg=Matrix([vector(w.subs(dic) for w in v.list()) for v in relAlexNeg])
SustNeg=EscalonarM2(relAlexSerieNeg,n)
print "The relations are used to write every one in terms of the basis inc" 
\end{lstlisting}

\begin{paso}\label{pasomorfismo}
The main point is to test if there is a homomorphism $\varphi:G_1\to G_2$
such that $x_k\in G_1$ is sent to
$x_k\prod_{(i,j)\in\mathcal{B}}[x_i,x_j]^{x_{k,i,j}}\bmod \tilde{\gamma}_4(G)$.
We express the image of $[x_i,x_j]$ as an element in~$M\otimes T$
(this is by far the most long computation!).
With this data, we compute the image
of the relations of $G_1$, which will be now
words in the truncated Alexander invariant $M_2\otimes T$.
 
We write these elements only in terms of $x_{i,j}$, $(i,j)\in\mathcal{B}$.
The next step is express the relations as linear combination
of $\sigma_k x_{i,j}$, $1\leq k\leq n$ and $(i,j)\in\mathcal{B}$,
with coefficients in~$S$.
\end{paso}

\begin{lstlisting}[breaklines=true]
TotImagenMorfismo=[]
for u in tot:
    v=ImagenMorfismo(u,XX,YY,S,LRv,dicv,inc)
    TotImagenMorfismo.append(v)
    if u[-1]==n:
        print "Finished pairs with line ", u[0]
        
relAlexVarPos=[]
for ser in relAlexSeriePos:
    vct=0
    for j in range(n*(n-1)/2):
        vct+=ser[j].change_ring(S)*TotImagenMorfismo[j]
    relAlexVarPos.append(vct)
relAlexVarPos=MatTrunc(Matrix(relAlexVarPos))
relAlexVarPos=MatTrunc(relAlexVarPos*SustNeg.change_ring(T).transpose())
print "Images of relations of first group in terms of the basis inc in the second group, done"

relZPos=[]
for prueba in relAlexVarPos:
    relZPos.append(vector(flatten([[pr1.derivative(v).constant_coefficient() for v in T.gens()] for pr1 in prueba])))
relZPos=Matrix(nrels,relZPos)
print "Integral matrix from relations"
\end{lstlisting}

\begin{paso}
In order to compute $M_2$ we need to add the relations given by Jacobi identities
as $\mathbb{Z}$-linear combinations in $\{\sigma_k x_{i,j}\mid 1\leq k\leq n,(i,j)\in\mathcal{B}\}$.
We quotient the previous free module by this relation. In this case we obtain 
a free abelian module; we pass from $23\times 11=253$ relations to $91$.
\end{paso}

\begin{lstlisting}[breaklines=true]
JCB=[]
for i in range(1,n-1):
    for j in range(i+1,n):
        for k in range(j+1,n+1):
            JCB+=[T0.gen(i-1)*XX[j,k].change_ring(T0)+T0.gen(j-1)*XX[k,i].change_ring(T0)+T0.gen(k-1)*XX[i,j].change_ring(T0)]
JCB=Matrix(JCB)
JCB1=MatTrunc(JCB*SustNeg.transpose())
JCB2=Matrix(JCB.nrows(),flatten([[v.derivative(w).constant_coefficient() for w in T0.gens()] for v in JCB1.list()]))
print "Simplified Jacobi relations as integral matrix done"
SF,U,V=JCB2.smith_form()
Jdiag=[SF[v,v] for v in range(SF.rank())]
if Set(Jdiag)!=Set([1]):
    print "Torsion at level 2"
print "Smith form of Jacobi relations, done", SF.ncols()-SF.rank()," generators left" 
\end{lstlisting}

\begin{paso}
We write down the relations taking into account Jacobi relations. 
The equations are all the coefficients in the basis. We solve
the linear equation of $930$~equations and $253$~unknowns.
The system has relations over~$\mathbb{Q}$ but not over~$\mathbb{Z}$
(the smallest ring is $\mathbb{Z}\left[\frac{1}{5}\right]$).  
\end{paso}

\begin{lstlisting}[breaklines=true]
eqs=(relZPos*V).matrix_from_columns([SF.rank()..JCB2.ncols()-1])
eqs1=[_ for _ in list(Set(eqs.list())) if _!=0]
print len(eqs1)," equations in ",len(inc)*n," unknowns."
Aeq=Matrix([vector(eq.monomial_coefficient(v) for v in S.gens()) for eq in eqs1])
Beq=vector(-eq.constant_coefficient() for eq in eqs1)
print "Ranks: ", Aeq.rank(),Aeq.augment(Beq).rank()
U1,U2,U3=Aeq.smith_form()
diageq=[U1[v,v] for v in range(U1.rank())]
diageq1=[_ for _ in diageq if _!=1]
print "Ones in the diagonal of the smith form: ",diageq.count(1)
print "Rest of the diagonal of the smith form: ",diageq1
print "Non integers in the Particular Solution:", [_ for _ in U1.solve_right(U2*Beq) if _ not in ZZ]
\end{lstlisting}

\begin{prt}
Some auxiliar functions.
\end{prt}

\begin{enumerate}[leftmargin=*]
\enet{(\texttt{S}\arabic{enumi})}
\item \texttt{Delta}
\begin{description}
\item[Input] A list of consecutive numbers $[i,\dots,j]$ and a braid group.
\item[Output] The half-twist braid involving the strands $i,\dots,j$. 
\end{description}
\begin{lstlisting}[breaklines=true]
def Delta(lista,B):
    res=B(1)
    l=copy(lista)[:-1]
    while len(l)>0:
        res=res*B(l)
        l=l[:-1]
    return res
\end{lstlisting}

\item \texttt{clean\_conj}
\begin{description}
\item[Input] A list of numbers $[i,\dots,j]$ and a free group. The list of numbers represents
the Tietze representation of an element~$x$ in the free group (with generators~$g_j$).
\item[Output] It returns \texttt{None} if $x$ is not conjugate to a generator. If it is,
it returns a list $[[i],[j_1,\dots,j_r]]$ such that $x=x_i^{x_{j_1\cdot\ldots\cdot x_{j_r}}}$. 
\end{description}
\begin{lstlisting}[breaklines=true]
def clean_conj(lista,F):
    n=len(lista)
    if n%2==0:
        print "Not conjugate to a generator"
        return None
    m=ZZ((n-1)/2)
    a=lista[m]
    b=lista[m+1:]
    if F(b)*F(lista[:m])!=F(1):
        print "Not conjugate to a generator"
        return None
    return [a,b]        
\end{lstlisting}
\item\label{S-relations} \texttt{list\_of\_relations\_2}
\begin{description}
\item[Input] A permutation to indicate the order of the lines in the beginning
and a wiring diagram, see \autoref{prt1}. For each crossing
we give a list with two members: the braid from the previous crossing, and 
the lines in the crossing. We assume a generic wiring diagram.
\item[Output] It returns the presentation of the group, as a list of elements with two entries;
the first one is a list $[i_1,\dots,i_r]$, $r\geq 2$; the second one is a list of
$r$ lists, which represent the Tietze representation of words $w_1,\dots,w_r$. This element
means that $x_{i_1}^{w_1}\cdot\ldots\cdot x_{i_r}^{w_r}$ commutes with 
$x_{i_1}^{w_1},\ldots, x_{i_r}^{w_r}$ (note that in other papers the reversed product is used).
\end{description}

\begin{lstlisting}[breaklines=true]
def list_of_relations_2(orden0,wiring0):
    wiring=[[B(_[0]),_[1]]for _ in wiring0]
    res1=[]
    trenza=B(1)
    orden=copy(orden0)
    for cruce in wiring:
        u,v=cruce
        orden=(u^-1).permutation()*orden
        trenza=trenza*u
        w=[orden.inverse()(_) for _ in v]
        res1.append([trenza,v,w,orden])
        trw=Delta(w,B)
        trenza=trenza*trw
        orden=trw.permutation().inverse()*orden
    final=[]
    for aa in res1:
        lis=[]
        for i in aa[2]:
            accion=[_.sign()*orden0(_.abs()) for _ in (F([i])*aa[0]^-1).Tietze()]
            listas=clean_conj(accion,F)
            if listas==None:
                print 'Problems'
                return None
            else:
                lis+=[listas]
        datos1=[[_[0] for _ in lis],[_[1] for _ in lis]]
        final+=[datos1]
    return final
\end{lstlisting}
\item \texttt{TietzeList}
\begin{description}
\item[Input] A word in the free group.
\item[Output] A list with its Tietze representation.
\end{description}

\begin{lstlisting}[breaklines=true]
def TietzeList(w):
    return list(w.Tietze())
\end{lstlisting}
\item \texttt{CommProd1}
\begin{description}
\item[Input] A number $a\in\{1,\dots,n\}$, a list representing a element~$x$ in the free group,
the free group $\mathbb{F}_n$ and the ring of Laurent polynomials in $n$ variables
and a dictionnary \texttt{X} associating to each $[i,j]$ an element in the free module
over the Laurent ring polynomial.
\item[Output] The element $[g_a,x]$ written as an element of the Alexander invariant~$M$.
\end{description}

\begin{lstlisting}[breaklines=true]
def CommProd1(a,l1,F,LR,XX):
    ll1=list(F(l1).Tietze())
    if len(ll1)==0:
        return M(0)
    return XX[a,ll1[0]]+tt([ll1[0]],LR)*CommProd1(a,ll1[1:],F,LR,XX)
\end{lstlisting}

\item \texttt{CommCyclic1}
\begin{description}
\item[Input] A list~$L$, like the output of \texttt{list\_of\_relations\_2},
the free group $\mathbb{F}_n$ and the ring of Laurent polynomials in $n$ variables,
a dictionnary \texttt{X} associating to each $[i,j]$ an element in the free module
over the Laurent ring polynomial
and a variable indicating the direction of the product yielding the local central
element (set by default to \texttt{RightLeft}).
\item[Output] The list of elements in the Alexander invariant induced by the relations
given by~$L$.
\end{description}

\begin{lstlisting}[breaklines=true]
def CommCyclic1(L,F,LR,XX,dir='RL'):
    res=[]
    r=len(L[0])
    comb=L[0]
    mn=comb.index(min(comb))
    L1=[v[mn:]+v[:mn] for v in L]
    LF=[F([-u for u in reversed(L1[1][j])]+[L1[0][j]]+L1[1][j]) for j in range(r)]
    for j in range(1,r):
        u=F(L1[1][j])
        if dir=='RL':
            LFp=[_ for _ in reversed(LF[j+1:]+LF[:j])]
        elif dir=='LR':
            LFp=LF[j+1:]+LF[:j]
        w=u*prod(LFp)*u^-1
        res+=[CommProd1(L1[0][j],list(w.Tietze()),F,LR,XX)]
    return res
\end{lstlisting}

\item \texttt{tt}
\begin{description}
\item[Input] A list of non-zero integers and a ring of Laurent polynomials in $n$ variables.
\item[Output] The monomial defined by the list.
\end{description}

\begin{lstlisting}[breaklines=true]
def tt(i,ring):
    return prod([LR.gen(ZZ(j).abs()-1)^(ZZ(j).sign()) for j in i])
\end{lstlisting}

\item \texttt{MatTrunc}
\begin{description}
\item[Input] A matrix with coefficient in a power series ring.
\item[Output] The matrix where each entry has been truncated to level~$2$.
\end{description}

\begin{lstlisting}[breaklines=true]
def MatTrunc(mat):
    A=mat
    m=A.nrows()
    return Matrix(m,[v.truncate(2) for v in A.list()])
\end{lstlisting}

\item \texttt{EscalonarM2}
\begin{description}
\item[Input] A matrix over the power series ring representing a list of homogeneous linear equations
with unknowns $x_{i,j}$  and the number of lines~$n$.
\item[Output] Let $\mathcal{B}$ as in \autoref{paso1}; 
each column of this matrix is the expression of the corresponding $x_{i,j}$
in terms of the elements~$x_{k,l}$, $(k,l)\in\mathcal{B}$ in the Alexander invariant over the power series ring.
\end{description}

\begin{lstlisting}[breaklines=true]
def EscalonarM2(matriz,n):
    U=matriz
    U0=Matrix(U.nrows(),[v.constant_coefficient() for v in U.list()])
    A,B=U0.echelon_form(transformation=True)
    U1=B*U
    Apivot=[]
    j=0
    for i in range(A.rank()):
        while A[i,j]==0:
            j=j+1
        Apivot.append(j)
    for i in range(A.rank()):
        j=Apivot[i]
        U1.rescale_row(i,U1[i,j]^-1)
        for k in range(i)+range(i+1,A.nrows()):
            U1.add_multiple_of_row(k,i,-U1[k,j])
    SustNeg=identity_matrix(T0,n*(n-1)/2)
    for i in range(A.rank()):
        for  j in range(n*(n-1)/2):
            SustNeg[j,Apivot[i]]=-U1[i,j].truncate(2).polynomial()
    SustNeg=SustNeg.delete_rows(Apivot)
    return SustNeg
\end{lstlisting}
\item \texttt{ImagenMorfismo}
\begin{description}
\item[Input] A list $[i,j]$ representing the commutator $x_{i,j}$, 
a dictionnary \texttt{XX} associating to each $[i,j]$ an element in the free module
over the Laurent ring polynomial, a dictionnary \texttt{YY} associating to each $[k,i,j]$ 
the unknown $x_{k,i,j}$, representing $[x_k,[x_i,x_j]]\equiv (t_k-1)x_{i,j}\equiv \sigma_kx_{i,j}$,
$S$ is a polynomial ring, \texttt{LRv} is the Laurent Ring with coefficients in the unknowns,
\texttt{dicv} is the evaluation $t_i\mapsto 1+\sigma_i$ and \texttt{inc} is as in \autoref{paso1}.
\item[Output] We consider the group of \autoref{pasomorfismo}.
This function uses it to express the image of $x_{i,j}$. 
\end{description}

\begin{lstlisting}[breaklines=true]
def ImagenMorfismo(L,XX,YY,S,LRv,dicv,inc):
    i,j=L
    res=XX[L].change_ring(S)
    ai=tt([i],LRv).subs(dicv)
    aj=tt([j],LRv).subs(dicv)
    aij=(ai*aj).truncate(2)
    bij=(ai-aij)
    cij=(aj-aij)
    res+=sum([(bij*YY[i,u,v]-cij*YY[j,u,v])*XX[u,v].change_ring(S) for (u,v) in inc])
    return res
\end{lstlisting}
\end{enumerate}

\end{document}